\documentclass[11pt]{article}
\usepackage[margin=1.05in]{geometry}
\usepackage[english]{babel}
\usepackage{microtype}
\usepackage{amsmath,amssymb,amsthm,mathtools,bm}
\usepackage{enumitem}
\usepackage{graphicx}
\usepackage{float}
\usepackage{booktabs}
\usepackage{xcolor}
\usepackage{algorithm}
\usepackage{algpseudocode}
\usepackage[colorlinks=true,linkcolor=blue!50!black,citecolor=green!40!black,urlcolor=blue!50!black]{hyperref}
\usepackage[capitalize,nameinlink]{cleveref}

\numberwithin{equation}{section}

\theoremstyle{plain}
\newtheorem{theor}{Theorem}[section]
\newtheorem{lema}[theor]{Lemma}
\newtheorem{prop}[theor]{Proposition}

\theoremstyle{definition}
\newtheorem{defi}[theor]{Definition}

\theoremstyle{remark}
\newtheorem{rmk}[theor]{Remark}

\crefname{theor}{Theorem}{Theorems}
\crefname{lema}{Lemma}{Lemmas}
\crefname{prop}{Proposition}{Propositions}
\crefname{coro}{Corollary}{Corollaries}
\crefname{defi}{Definition}{Definitions}
\crefname{assum}{Assumption}{Assumptions}
\crefname{rmk}{Remark}{Remarks}
\crefname{algorithm}{Algorithm}{Algorithms}

\newcommand{\R}{\mathbb{R}}
\newcommand{\Uad}{\mathcal{U}_{\mathrm{ad}}}
\newcommand{\cU}{\mathcal{U}}
\newcommand{\dd}{\,\mathrm{d}}
\newcommand{\des}{\mathrm{des}}
\DeclareMathOperator{\Lip}{Lip}

\begin{document}

\title{Stochastic Gradient Descent for Traffic and Pedestrian Models via Neural Networks and Parameter Identification}

\author{Edward Lester\thanks{Department of Mathematics and Statistics, San Diego State University, San Diego, CA 92182, USA; email: elester0720@sdsu.edu}
\and
Dao Nguyen\thanks{Corresponding author. Department of Mathematics and Statistics, San Diego State University, San Diego, CA 92182, USA; email: dnguyen28@sdsu.edu}}

\date{}
\maketitle
\thispagestyle{empty}

\begin{abstract}
We study the calibration of interaction-driven ordinary differential equation (ODE) models for traffic and pedestrian dynamics by means of optimal control and stochastic gradient descent. The interaction forces are either prescribed by classical physics-inspired laws, whose parameters are identified from trajectory data, or represented by feed-forward neural networks whose weights act as control variables. For the neural-network-driven particle system we prove well-posedness of the state equation, Lipschitz-continuous dependence of the trajectories on the initial data and on the network parameters, existence of an optimal parameter set, and a first-order optimality system with an explicit backward adjoint equation. The resulting adjoint representation of the gradient is combined with a projected mini-batch ADADELTA scheme, so that each gradient evaluation requires one forward and one backward ODE solve, independently of the number of network parameters. Building on this framework we propose two new second-order models, one for vehicular traffic and one for crowds, that combine goal-seeking, learned interaction, repulsion--alignment or dissipation, and external control terms. Numerical experiments compare six models (a follow-the-leader model of Lighthill--Whitham--Richards type, a neural-network traffic model, the social force model, a neural-network pedestrian model, and the two new models) in terms of cost reduction and root-mean-square trajectory deviation. The neural-network pedestrian model achieves the lowest final cost, while the richer second-order models trade a higher calibration cost for individualised, behaviourally interpretable dynamics. We conclude by discussing limitations and directions for incorporating richer behavioural and environmental effects.
\end{abstract}



%======================================================================
\section{Introduction}\label{sec:intro}
%======================================================================

This paper explores methodologies for the calibration of interaction models for particle systems, with direct focus on applications in pedestrian dynamics and traffic flow. Traditional models, such as the Lighthill--Whitham--Richards (LWR) traffic model \cite{LighthillWhitham1955,Richards1956} and the social force model for crowd motion \cite{HelbingMolnar1995}, rely on explicit physical assumptions and predefined interaction forces. While these models have proven effective, their limitations become apparent when attempting to replicate real-world behaviour across diverse scenarios. The variability of human behaviour and of environmental conditions often challenges the assumptions built into classical models, which motivates more flexible, data-driven approaches that can adapt to complex and nonlinear interaction patterns. The relation between microscopic follow-the-leader dynamics and the macroscopic LWR model provides an important example of how agent-based and continuum descriptions complement one another \cite{HoldenRisebro2018}, and second-order car-following models such as the optimal velocity and intelligent driver models \cite{Bando1995,Treiber2000} show how driver intent and relative speed can be incorporated. Likewise, the social force model remains a foundational microscopic description of pedestrian interactions \cite{HelbingMolnar1995}.

From a broader perspective, pedestrian and traffic systems belong to the class of interacting particle systems and collective dynamics. Kinetic descriptions provide a natural bridge between microscopic agent-based models and macroscopic continuum equations \cite{PareschiToscani2013,Toscani2006,Carrillo2010}. Related frameworks have been applied to swarming, flocking, opinion formation, and other forms of collective behaviour \cite{CastellanoFortunatoLoreto2009,CuckerSmale2007}. This broader mathematical setting motivates interaction-driven learning algorithms that extract effective dynamical laws from experimental and trajectory data.

Pedestrian dynamics has attracted considerable attention because of its importance in urban planning, evacuation, autonomous navigation, and public safety. Classical microscopic models describe pedestrian motion through attractive and repulsive forces intended to mimic human interactions \cite{HelbingMolnar1995}; comprehensive surveys can be found in \cite{Schadschneider2018,Bellomoetal2022}. Data-driven prediction methods based on artificial neural networks have demonstrated promising performance in estimating pedestrian velocities directly from observed trajectories \cite{TordeuxEtAl2019,GottlichKnapp2020}. These developments illustrate both the value of classical modelling principles and the need for flexible learning mechanisms that account for behavioural variability not easily represented by a fixed analytical force law.

Parameter identification has therefore become a fundamental topic in the calibration of dynamical systems arising from traffic flow, pedestrian motion, and interacting particle systems. Instead of prescribing every interaction parameter a priori, recent research recovers unknown parameters or interaction kernels from trajectory data. Bayesian approaches have been proposed for estimating parameters of social force models while simultaneously quantifying uncertainty in the recovered values \cite{CorbettaEtAl2015,Stuart2010}. Similar inverse problems have been investigated for macroscopic pedestrian dynamics, where microscopic trajectory measurements are used to estimate continuum-model parameters \cite{GomesStuartWolfram2019}. These developments demonstrate the increasing importance of combining mathematical modelling with optimisation, inverse-problem techniques, and data assimilation.

Beyond parameter estimation, recent years have witnessed the emergence of data-driven approaches that directly infer interaction laws from observed trajectories. Rather than assuming a predefined force field, these methods aim to reconstruct the interaction kernels governing collective behaviour using machine learning and statistical techniques. Lu et al.\ \cite{Lu2019,LuMaggioniTang2021} developed a nonparametric framework for learning interaction kernels in heterogeneous multi-agent systems from multiple trajectories and established theoretical conditions for learnability. Similar ideas have been explored for discovering emergent collective dynamics from observational data \cite{Zhong2020}. Such approaches provide considerably greater flexibility than traditional physics-based calibration methods while preserving interpretable interaction mechanisms.

Advances in machine learning allow neural networks to learn complex nonlinear functions from data, providing a means to approximate forces and behaviours that are difficult to capture with traditional physics-based models. This paper builds upon these developments, following \cite{GottlichTotzeck2022}, by using a neural-network-based framework to model the forces that direct agent interactions. Instead of manually defining the form of every interaction force, the approach allows the neural network to learn directly from empirical data. This data-driven modelling can improve the realism and adaptability of interaction models, making them more suitable for simulation and prediction in dynamic environments.

Continuous-depth neural networks have further strengthened the connection between machine learning and dynamical systems. Neural ODEs interpret deep neural networks as continuous dynamical systems whose evolution is governed by differential equations \cite{ChenEtAl2018}; computing their gradients with the adjoint method keeps the memory cost constant with respect to the depth of the network, while allowing adaptive numerical integration. Stable architectures inspired by ODEs have been proposed to mitigate exploding and vanishing gradients during training \cite{HaberRuthotto2017}, physics-informed neural networks combine differential-equation structure with data-driven approximation \cite{RaissiPerdikarisKarniadakis2019}, and structure-preserving neural ODEs have been used for learning and control of mechanical systems \cite{DuongEtAl2024}. Together, these developments motivate the use of neural differential models for system identification, parameter calibration, and control.

A key challenge in this endeavour is the calibration of the neural-network parameters to fit observed data. To address this challenge, we frame the calibration task as an optimal control problem in which the network parameters are treated as control variables that influence the overall dynamics of the system. By defining a suitable cost functional that measures the discrepancy between the predicted system trajectories and the data, the problem becomes one of finding the control parameters that minimise this discrepancy. Optimisation subject to differential-equation constraints provides the appropriate mathematical setting for adjoint-based gradient computation and large-scale parameter estimation \cite{HinzeEtAl2009}. Gradient-based methods are particularly natural in this setting, and stochastic gradient methods \cite{RobbinsMonro1951,BottouCurtisNocedal2018} with adaptive step sizes such as ADADELTA \cite{Zeiler2012} and Adam \cite{KingmaBa2015} have become standard tools for training high-dimensional models. To solve the calibration problem, we employ an adjoint-based, projected, mini-batch stochastic gradient algorithm that computes the gradients needed for the optimisation.

The significance of this work lies in its potential to improve predictive models through data-driven machine learning. The framework is designed to be validated on real-world trajectory datasets \cite{trafficdynamics,Pedynamics}, with the aim of producing more realistic, adaptable, and predictive models of collective agent behaviour. Reinforcing this strategy with empirical data for these scenarios can establish the applicability of neural-network-based parameter calibration. The approach not only broadens the scope of earlier work, but also lends itself to more sophisticated modelling with a wider range of applications.

This paper adopts an approach grounded in prior research, drawing upon established results that form the foundation for much of the notation and many of the formulas used throughout \cite{GottlichTotzeck2022,DuongEtAl2024}. In particular, these earlier works provide critical insights into the structure of the optimal control problem and into strategies for parameter calibration. The framework and ideas derived from these works serve as building blocks for the new models presented here. They also informed the design of the computer code that was written to generate the visualisations and results supporting the analysis in this work. Furthermore, an established definition of the cost is used to measure the distance between the observed data and the simulation results for each model.

Although parameter estimation, interaction-kernel learning, neural ODEs, and data-driven pedestrian modelling have each received substantial attention, they are usually studied separately or for a single application. A central objective of this work is to bring these components together in one computational framework that treats both traffic and pedestrian dynamics, regards classical parameter identification and neural-network interaction learning as complementary calibration strategies, and compares them under a common cost functional and a common optimisation method. The main contributions are as follows.
\begin{itemize}[leftmargin=2em,itemsep=0.3em]
\item We formulate the calibration of interaction-driven ODE systems as an optimal control problem in which model parameters, including neural-network weights, act as controls (\cref{sec:ocp}).
\item For the neural-network-driven system we give self-contained proofs of well-posedness, of Lipschitz dependence on data and parameters with explicit constants, of existence of a minimiser, and of a first-order optimality system with an explicit backward adjoint equation (\cref{thm:wellposed,thm:continuousdependence,thm:existence,thm:optimality}).
\item We derive an adjoint-based projected mini-batch ADADELTA algorithm whose per-iteration cost is independent of the number of parameters (\cref{sec:sgd}, \cref{alg:sgd}).
\item We propose a new second-order traffic model and a new crowd model that combine goal seeking, learned interactions, repulsion and alignment or dissipation, and external control (\cref{sec:newtraffic,sec:newcrowd}).
\item We compare six traffic and pedestrian models numerically, quantify cost reduction and trajectory deviation, and identify limitations and concrete directions for future work (\cref{sec:discussion,sec:future}).
\end{itemize}

The remainder of the paper is organised as follows. \Cref{sec:prelim,sec:ocp,sec:sgd} provide an in-depth treatment of the optimal control problem and of the design of the artificial neural network. \Cref{sec:prelim} collects definitions of the relevant terms and the auxiliary results used later. \Cref{sec:ocp} introduces the basic framework of the ODE system, which formalises how individual agents move under the influence of interaction forces, and establishes its analysis. Identifying each element precisely is vital, since we maintain a consistent basis that is referred to and expanded upon in the later sections. \Cref{sec:sgd} introduces gradient descent, primarily the stochastic gradient algorithm, which is used for all models in this paper.

\Cref{sec:traffic} shifts the focus to traffic dynamics. Established models are critically examined with respect to their limitations, advantages, and practical applications, and potential areas for improvement are identified. To ground the discussion, coded implementations have been created and plots of the results are included, offering an applied perspective on the behaviour and performance of these models. Following this, a new traffic model is introduced that synthesises the attributes of the previous models with new features. Its results are contrasted with those of the established models, and the wider question of model selection is explored.

\Cref{sec:crowd} centres on models of crowd dynamics. As in the previous section, two established models are reviewed and discussed together with plots of example simulations. Once more a new model is put forward, and the benefits and drawbacks it possesses, individually and in comparison to the other models, are highlighted. \Cref{sec:discussion} compares all six models in a common table.

\Cref{sec:future} expands on the limitations of our models and identifies multiple areas of improvement. Potential paths for model extension and refinement are discussed, offering a road map for continued research in optimal control and complex system dynamics; besides the attributes already present, entirely new aspects are brought forward that would increase the realism of the models in comparison with real-world data. \Cref{sec:conclusion} concludes the paper, and additional cost tables are collected in \cref{app:tables}.

%======================================================================
\section{Preliminaries}\label{sec:prelim}
%======================================================================

Throughout, $|\cdot|$ denotes the Euclidean norm on $\R^m$ and the induced operator (or Frobenius) norm on matrices, and $\langle\cdot,\cdot\rangle$ the Euclidean inner product. For $T>0$ we use the Lebesgue space $L^2(0,T;\R^m)$ and the Sobolev space $H^1(0,T;\R^m)$ with norm
\[
\|f\|_{H^1(0,T)}=\Big(\int_0^T|f(t)|^2\dd t+\int_0^T|f'(t)|^2\dd t\Big)^{1/2},
\]
which measures both size and smoothness. Recall the continuous embedding $H^1(0,T;\R^m)\hookrightarrow C([0,T];\R^m)$.

A map $f:\R^m\to\R^p$ is \emph{globally Lipschitz} if there is $L\ge 0$ with $|f(x)-f(y)|\le L|x-y|$ for all $x,y$; the smallest such $L$ is denoted $\Lip(f)$. For normed spaces $V,W$ and an open set $O\subset V$, a map $f:O\to W$ is \emph{G\^ateaux differentiable} at $x$ if the directional derivative $f'(x)[h]=\lim_{\epsilon\to0}\epsilon^{-1}(f(x+\epsilon h)-f(x))$ exists for every $h\in V$ and is linear and bounded in $h$; it is \emph{Fr\'echet differentiable} at $x$ if there is a bounded linear operator $B:V\to W$ with $\|f(x+h)-f(x)-Bh\|_W=o(\|h\|_V)$ as $\|h\|_V\to 0$.

We shall use the following classical results.

\begin{theor}[Picard--Lindel\"of]\label{thm:PL}
Let $F:[0,T]\times\R^m\to\R^m$ be continuous and globally Lipschitz in its second argument, uniformly in $t$. Then for every $y_0\in\R^m$ the problem $\dot y=F(t,y)$, $y(0)=y_0$, has a unique solution $y\in C^1([0,T];\R^m)$.
\end{theor}

\begin{theor}[Carath\'eodory, linear case]\label{thm:Cara}
Let $A\in L^\infty(0,T;\R^{m\times m})$, $r\in L^2(0,T;\R^m)$ and $r_0\in\R^m$. Then $\dot h=A(t)h+r(t)$, $h(0)=r_0$, has a unique absolutely continuous solution, and $h\in H^1(0,T;\R^m)$.
\end{theor}

\begin{lema}[Gronwall]\label{lem:gronwall}
Let $\phi\ge0$ be continuous on $[0,T]$ and satisfy $\phi(t)\le C+\int_0^t\alpha(s)\phi(s)\dd s$ with $C\ge0$ and continuous $\alpha\ge0$. Then $\phi(t)\le C\exp\big(\int_0^t\alpha(s)\dd s\big)$ for all $t\in[0,T]$.
\end{lema}

Finally, we recall a standard first-order optimality result for equality-constrained optimisation in Banach spaces \cite[Cor.~1.3]{HinzeEtAl2009}.

\begin{lema}[Existence of an adjoint state]\label{lem:hinze}
Let $Y,U,Z$ be Banach spaces, $J:Y\times U\to\R$ and $e:Y\times U\to Z$, and let $(\bar y,\bar u)$ be a local solution of
\[
\min_{(y,u)\in Y\times U}J(y,u)\quad\text{subject to}\quad e(y,u)=0,\quad u\in\Uad.
\]
Assume that
\begin{enumerate}[label=\textup{(A\arabic*)},leftmargin=2.6em]
\item $\Uad\subset U$ is nonempty, convex and closed;
\item $J$ and $e$ are continuously Fr\'echet differentiable;
\item there is a neighbourhood $V$ of $\Uad$ such that for every $u\in V$ the equation $e(y,u)=0$ has a unique solution $y=y(u)\in Y$;
\item $e_y(y(u),u)\in\mathcal L(Y,Z)$ has a bounded inverse for every $u\in V$.
\end{enumerate}
Then there exists an adjoint state $\bar p\in Z^*$ such that
\[
e(\bar y,\bar u)=0,\qquad e_y(\bar y,\bar u)^*\bar p=-J_y(\bar y,\bar u),\qquad
\langle J_u(\bar y,\bar u)+e_u(\bar y,\bar u)^*\bar p,\,u-\bar u\rangle_{U^*,U}\ge0\quad\forall u\in\Uad.
\]
\end{lema}

%======================================================================
\section{The optimal control problem}\label{sec:ocp}
%======================================================================

\subsection{Interacting particle system driven by neural networks}

We consider $N$ agents with states $x_i(t)\in\R^n$, $i=1,\dots,N$, and write $x=(x_1,\dots,x_N)\in\R^{Nn}$. The interaction between agents is described by feed-forward neural networks.

\begin{defi}[Feed-forward neural network]\label{def:NN}
Let $L\ge 2$ and widths $n^{(1)},\dots,n^{(L)}\in\mathbb N$. A feed-forward network with $L$ layers maps an input $r\in\R^{n^{(1)}}$ to an output $\mathcal N_u(r)=a^{(L)}\in\R^{n^{(L)}}$ as follows.
\begin{itemize}
\item Input layer: $a^{(1)}_1=1$ (bias unit) and $a^{(1)}_k=r_{k-1}$ for $k=2,\dots,n^{(1)}+1$.
\item Hidden layers $l=2,\dots,L-1$: $a^{(l)}_1=1$ and
$a^{(l)}_k=g^{(l)}\Big(\sum_{j=1}^{n^{(l-1)}+1}u^{(l-1)}_{j,k}a^{(l-1)}_j\Big)$ for $k=2,\dots,n^{(l)}+1$.
\item Output layer (no bias unit): $a^{(L)}_k=g^{(L)}\Big(\sum_{j=1}^{n^{(L-1)}+1}u^{(L-1)}_{j,k}a^{(L-1)}_j\Big)$ for $k=1,\dots,n^{(L)}$.
\end{itemize}
Here $g^{(l)}:\R\to\R$ are activation functions applied componentwise and $u^{(l)}\in\R^{(n^{(l)}+1)\times n^{(l+1)}}$ contains the weights connecting layer $l$ (including its bias unit) to layer $l+1$; $u^{(l)}_{j,k}$ is the weight from neuron $a^{(l)}_j$ to $a^{(l+1)}_k$ (in the index shift of the hidden-layer formula, $u^{(l-1)}_{j,k}$ connects $a^{(l-1)}_j$ to $a^{(l)}_k$). All weights are stacked into a single vector $u\in\R^K$, $K=\sum_{l=1}^{L-1}(n^{(l)}+1)\,n^{(l+1)}$.
\end{defi}

The neural interaction law between agents $i$ and $j$ is a network $W^{i,j}_u:\R^n\to\R^n$ in the sense of \cref{def:NN} with $n^{(1)}=n^{(L)}=n$. In the most common situation all pairs share one network, $W_u^{i,j}=\mathcal N_u$, as in \cite{GottlichTotzeck2022}; more generally, the weights of $W^{i,j}_u$ form a sub-vector of $u$. All results below hold in either case. The state $x=x(u)$ is governed by
\begin{equation}\label{eq:ode}
\frac{\dd}{\dd t}x_i(t)=\sum_{j=1}^N W_u^{i,j}\big(x_j(t)-x_i(t)\big),\qquad x_i(0)=z_i^0,\qquad i=1,\dots,N,
\end{equation}
which we abbreviate as $\dot x=F(x,u)$, $x(0)=z^0$, with $F_i(x,u)=\sum_{j=1}^N W_u^{i,j}(x_j-x_i)$.

Given reference trajectories $z=(z_1,\dots,z_N)\in L^2(0,T;\R^{Nn})$, for instance measured positions, the discrepancy between model and data is measured by
\begin{equation}\label{eq:cost}
J(x;z)=\frac12\int_0^T|z(t)-x(t)|^2\dd t=\frac12\sum_{i=1}^N\int_0^T|z_i(t)-x_i(t)|^2\dd t=\frac12\|z-x\|^2_{L^2(0,T;\R^{Nn})}.
\end{equation}
With the admissible set $\Uad:=[-1,1]^K$ the parameter identification problem reads
\begin{equation}\label{E:3}
\min_{u\in\Uad}\ \hat J(u):=J(x(u);z)\qquad\text{subject to \eqref{eq:ode}}.
\end{equation}
We refer to $\hat J$ as the reduced cost functional and, when no ambiguity arises, simply write $x(t)$ for $x(u)(t)$. To interpret the size of the cost we use the root-mean-square deviation per agent and unit time,
\begin{equation}\label{eq:rmsd}
\bar d:=\sqrt{\frac{2J}{N\,T}}=\Big(\frac{1}{NT}\sum_{i=1}^N\int_0^T|z_i-x_i|^2\dd t\Big)^{1/2},
\end{equation}
which carries the physical unit of the positions (metres in our experiments).

\begin{rmk}
Since both $z$ and $x(u)$ belong to $L^2(0,T;\R^{Nn})$, so does their difference, and the norm in \eqref{eq:cost} is well defined. Explicitly, since $z(t)=(z_1(t),\dots,z_N(t))$ and $x(t)=(x_1(t),\dots,x_N(t))$ with $z_i(t),x_i(t)\in\R^n$,
\[
\|z-x(u)\|^2_{L^2(0,T;\R^{Nn})}=\int_0^T|z(t)-x(t)|^2_{\R^{Nn}}\dd t,\qquad |z(t)-x(t)|^2_{\R^{Nn}}=\sum_{i=1}^N|z_i(t)-x_i(t)|^2_{\R^n},
\]
so that
\[
J(x(u);z)=\frac12\int_0^T|z(t)-x(t)|^2_{\R^{Nn}}\dd t=\frac12\|z-x(u)\|^2_{L^2(0,T;\R^{Nn})}.
\]
Minimising $J$ with respect to $u$ therefore yields the parameters that best fit the data over the full time interval.
\end{rmk}

\subsection{Well-posedness and stability}

\begin{lema}[Regularity of the network]\label{lem:NN}
Let all activation functions $g^{(l)}$, $l=2,\dots,L$, be globally Lipschitz and let $\cU_R:=\{u\in\R^K:|u|\le R\}$. Then there is a constant $C_{\mathcal N}\ge1$, depending only on $R$, the widths, $\max_l\Lip(g^{(l)})$ and $\max_l|g^{(l)}(0)|$, such that for all $r,s\in\R^{n^{(1)}}$ and $u,\bar u\in\cU_R$:
\begin{align}
|\mathcal N_u(r)-\mathcal N_u(s)|&\le C_{\mathcal N}\,|r-s|,\label{eq:NN1}\\
|\mathcal N_u(r)|&\le C_{\mathcal N}\,(1+|r|),\label{eq:NN2}\\
|\mathcal N_u(r)-\mathcal N_{\bar u}(r)|&\le C_{\mathcal N}\,(1+|r|)\,|u-\bar u|.\label{eq:NN3}
\end{align}
If, in addition, $g^{(l)}\in C^1(\R)$ with $(g^{(l)})'$ bounded and locally Lipschitz, then $(r,u)\mapsto\mathcal N_u(r)$ is continuously differentiable on $\R^{n^{(1)}}\times\R^K$ with locally Lipschitz derivative.
\end{lema}

\begin{proof}
Write $\hat a^{(l)}(r;u)$ for the vector of all neurons of layer $l$ including the bias unit, so that $\hat a^{(1)}=(1,r)$ and $\hat a^{(l+1)}=(1,g^{(l+1)}((u^{(l)})^{\top}\hat a^{(l)}))$ (without the leading $1$ for $l+1=L$). Let $\Lambda:=\max_l\Lip(g^{(l)})$ and note $|u^{(l)}|\le|u|\le R$. Since $g^{(l+1)}$ acts componentwise, it is Lipschitz with constant $\Lambda$ as a vector map and satisfies $|g^{(l+1)}(\xi)|\le c_0+\Lambda|\xi|$ with $c_0$ depending on $|g^{(l+1)}(0)|$ and the widths.

\emph{(i)} $|\hat a^{(1)}(r)-\hat a^{(1)}(s)|=|r-s|$ and $|\hat a^{(l+1)}(r)-\hat a^{(l+1)}(s)|\le\Lambda R\,|\hat a^{(l)}(r)-\hat a^{(l)}(s)|$; induction gives \eqref{eq:NN1}.

\emph{(ii)} $|\hat a^{(1)}(r)|\le 1+|r|$ and $|\hat a^{(l+1)}(r)|\le 1+c_0+\Lambda R\,|\hat a^{(l)}(r)|$; induction gives $|\hat a^{(l)}(r;u)|\le C_l(1+|r|)$ for all $l$, hence \eqref{eq:NN2}.

\emph{(iii)} Let $\delta_l:=|\hat a^{(l)}(r;u)-\hat a^{(l)}(r;\bar u)|$. Then $\delta_1=0$ and
\[
\delta_{l+1}\le\Lambda\big(|u^{(l)}|\,\delta_l+|u^{(l)}-\bar u^{(l)}|\,|\hat a^{(l)}(r;\bar u)|\big)\le\Lambda\big(R\,\delta_l+C_l(1+|r|)\,|u-\bar u|\big),
\]
so that by induction $\delta_l\le\tilde C_l(1+|r|)|u-\bar u|$, which is \eqref{eq:NN3}.

The last statement follows from the chain rule, since $\mathcal N_u(r)$ is a finite composition of bilinear maps $(u^{(l)},a)\mapsto(u^{(l)})^\top a$ and $C^1$ componentwise activations with locally Lipschitz derivatives.
\end{proof}

\begin{theor}[Well-posedness of the state equation]\label{thm:wellposed}
Assume that all activation functions $g^{(l)}$, $l=2,\dots,L$, are globally Lipschitz. Then, for every $u\in\R^K$ and every initial datum $z^0\in\R^{Nn}$, system \eqref{eq:ode} admits a unique solution $x\in C^1([0,T];\R^{Nn})$.
\end{theor}

\begin{proof}
Fix $u$ and choose $R\ge|u|$. By \eqref{eq:NN1}, each $W^{i,j}_u$ is globally Lipschitz with constant $C_{\mathcal N}$. For $x,y\in\R^{Nn}$,
\[
|F_i(x,u)-F_i(y,u)|\le C_{\mathcal N}\sum_{j=1}^N\big(|x_j-y_j|+|x_i-y_i|\big)\le 2NC_{\mathcal N}\,|x-y|,
\]
and summing the squares over $i$ gives $|F(x,u)-F(y,u)|\le L_F|x-y|$ with $L_F:=2N^{3/2}C_{\mathcal N}$. Hence $F(\cdot,u)$ is globally Lipschitz and continuous, and \cref{thm:PL} yields a unique solution $x\in C^1([0,T];\R^{Nn})$.
\end{proof}

\begin{rmk}
\Cref{thm:wellposed} states that the learned dynamics are deterministic and well defined: whenever the activations have bounded slope (as for $\tanh$, the logistic function, ReLU or softplus), the network-driven model produces exactly one trajectory for each initial configuration and each parameter vector. In particular, the forward simulations used inside the optimisation never break down in finite time.
\end{rmk}

\begin{theor}[Continuous dependence on data and parameters]\label{thm:continuousdependence}
Let the assumptions of \cref{thm:wellposed} hold, let $R,R_0>0$, and let $x,\bar x$ be the solutions of \eqref{eq:ode} corresponding to $(x_0,u)$ and $(\bar x_0,\bar u)$ with $|x_0|,|\bar x_0|\le R_0$ and $u,\bar u\in\cU_R$. Then there is a constant $M>0$, depending only on $T$, $N$, $R$, $R_0$, the network architecture and the activation functions, such that
\[
\|x-\bar x\|_{H^1(0,T;\R^{Nn})}\le M\big(|x_0-\bar x_0|+|u-\bar u|\big).
\]
\end{theor}

\begin{proof}
\emph{A priori bound.} By \eqref{eq:NN2}, $|F(0,u)|\le N^{3/2}C_{\mathcal N}$, so $|F(y,\bar u)|\le N^{3/2}C_{\mathcal N}+L_F|y|$. Integrating the equation for $\bar x$ and applying \cref{lem:gronwall} gives
\[
|\bar x(t)|\le R_x:=\big(R_0+TN^{3/2}C_{\mathcal N}\big)e^{L_FT},\qquad t\in[0,T].
\]
\emph{Parameter perturbation.} By \eqref{eq:NN3},
\[
|F(\bar x(t),u)-F(\bar x(t),\bar u)|\le\sqrt N\sum_{i,j}C_{\mathcal N}\big(1+|\bar x_j(t)-\bar x_i(t)|\big)|u-\bar u|\le L_u|u-\bar u|,
\]
with $L_u:=N^{5/2}C_{\mathcal N}(1+2R_x)$.
\emph{Gronwall argument.} Let $e:=x-\bar x$. Then
\[
|\dot e(t)|=|F(x,u)-F(\bar x,\bar u)|\le|F(x,u)-F(\bar x,u)|+|F(\bar x,u)-F(\bar x,\bar u)|\le L_F|e(t)|+L_u|u-\bar u|.
\]
Integration and \cref{lem:gronwall} yield $|e(t)|\le\big(|e(0)|+TL_u|u-\bar u|\big)e^{L_FT}=:E$ for all $t\in[0,T]$, and consequently $|\dot e(t)|\le L_FE+L_u|u-\bar u|$. Therefore
\[
\|e\|^2_{H^1}=\int_0^T|e|^2\dd t+\int_0^T|\dot e|^2\dd t\le T\big(E^2+(L_FE+L_u|u-\bar u|)^2\big),
\]
and the claim follows with $M:=\sqrt{T}\,\big((1+L_F)(1+TL_u)e^{L_FT}+L_u\big)$.
\end{proof}

\begin{rmk}
\Cref{thm:continuousdependence} expresses robustness of the model: small errors in the measured initial positions or small changes of the network weights cause only proportionally small changes of the simulated trajectories, uniformly on the time horizon. Note that only the Lipschitz continuity of the activations is needed; in contrast to earlier formulations no boundedness of the activations is required.
\end{rmk}

\subsection{Existence of an optimal parameter set}

\begin{theor}[Existence of a minimiser]\label{thm:existence}
Let the assumptions of \cref{thm:wellposed} hold and let $z\in L^2(0,T;\R^{Nn})$. Then problem \eqref{E:3} admits at least one minimiser $u^*\in\Uad$.
\end{theor}

\begin{proof}
Since $\Uad=[-1,1]^K$ is nonempty, closed and bounded in $\R^K$, it is compact. By \cref{thm:continuousdependence} with $R=\sqrt K$, the control-to-state map $u\mapsto x(u)$ is Lipschitz continuous from $\Uad$ into $H^1(0,T;\R^{Nn})$ and hence into $L^2(0,T;\R^{Nn})$. The functional $x\mapsto J(x;z)$ is continuous on $L^2$, so $\hat J$ is continuous on the compact set $\Uad$. Let $(u_k)\subset\Uad$ be a minimising sequence, $\hat J(u_k)\to m:=\inf_{\Uad}\hat J\ge0$. By compactness a subsequence converges to some $u^*\in\Uad$, and continuity gives $\hat J(u^*)=m$.
\end{proof}

\begin{rmk}\label{rmk:existence}
(i) Because the control space is finite dimensional, weak and strong convergence coincide; the weak-continuity assumptions on $u\mapsto W_u^{i,j}$ used in some infinite-dimensional formulations are automatically satisfied here. (ii) The minimiser need not be unique: permutations of hidden neurons and sign symmetries of odd activations leave $\mathcal N_u$ invariant, so $\hat J$ is in general non-convex with many equivalent minimisers. This is one of the motivations for the stochastic, adaptive optimisation scheme of \cref{sec:sgd}.
\end{rmk}

\subsection{Adjoint calculus and first-order optimality}

We first record the abstract adjoint representation of the reduced gradient.

\begin{prop}\label{prop:adjoint-gradient}
Let $Y,Z$ be real Hilbert spaces, $U=\R^K$, and let $e:Y\times U\to Z$ be continuously Fr\'echet differentiable. Suppose that for every $u\in U$ the state $S(u)\in Y$ is the unique solution of $e(S(u),u)=0$ and that $e_y(S(u),u)\in\mathcal L(Y,Z)$ is an isomorphism. Then $S$ is continuously Fr\'echet differentiable and the functional $\mathcal J(u):=\frac12\|z-S(u)\|_Y^2$, $z\in Y$, satisfies
\[
\nabla\mathcal J(u)=e_u(S(u),u)^*p(u),
\]
where $p(u)\in Z$ is the unique solution of the adjoint equation $e_y(S(u),u)^*p(u)=z-S(u)$.
\end{prop}

\begin{proof}
Differentiability of $S$ follows from the implicit function theorem. Let $s\in U$ be arbitrary. Since
\[
\mathcal J(u)=\frac12\langle z-S(u),z-S(u)\rangle_Y,
\]
the chain rule gives
\[
\mathcal J'(u)[s]=\langle S(u)-z,S'(u)[s]\rangle_Y.
\]
By definition of the Hilbert-space adjoint of $S'(u):U\to Y$,
\[
\langle S'(u)[s],\varphi\rangle_Y=\langle s,S'(u)^*\varphi\rangle_U\qquad\text{for all }s\in U,\ \varphi\in Y.
\]
Taking $\varphi=S(u)-z$ we obtain $\mathcal J'(u)[s]=\langle s,S'(u)^*(S(u)-z)\rangle_U$, and hence
\begin{equation}\label{eq:reduced-gradient}
\nabla\mathcal J(u)=S'(u)^*(S(u)-z).
\end{equation}
Differentiating the identity $e(S(u),u)=0$ in the direction $s$ yields
\[
e_y(S(u),u)S'(u)[s]+e_u(S(u),u)[s]=0,
\]
and therefore
\[
S'(u)=-e_y(S(u),u)^{-1}e_u(S(u),u),\qquad S'(u)^*=-e_u(S(u),u)^*\big(e_y(S(u),u)^*\big)^{-1}.
\]
Since $e_y(S(u),u)$ is an isomorphism, so is its adjoint, and the adjoint state
\[
p(u)=\big(e_y(S(u),u)^*\big)^{-1}(z-S(u))
\]
is well defined. Substituting into \eqref{eq:reduced-gradient} gives
\[
\begin{aligned}
\nabla\mathcal J(u)&=-e_u(S(u),u)^*\big(e_y(S(u),u)^*\big)^{-1}(S(u)-z)\\
&=e_u(S(u),u)^*\big(e_y(S(u),u)^*\big)^{-1}(z-S(u))=e_u(S(u),u)^*p(u).
\end{aligned}
\]
Consequently $\mathcal J'(u)[s]=\langle e_u(S(u),u)^*p(u),s\rangle_U$ for every direction $s\in U$.
\end{proof}

For the particle system \eqref{eq:ode} the adjoint equation can be written down explicitly. We use the notation
\[
A(t):=D_xF(x(t),u)\in\R^{Nn\times Nn},\qquad B(t):=D_uF(x(t),u)\in\R^{Nn\times K},
\]
so that, with $DW^{i,j}_u$ denoting the Jacobian of $W^{i,j}_u$ with respect to its argument and $\partial_uW^{i,j}_u$ the Jacobian with respect to the parameters,
\[
[A(t)h]_i=\sum_{j=1}^N DW^{i,j}_u\big(x_j-x_i\big)(h_j-h_i),\qquad [B(t)q]_i=\sum_{j=1}^N\partial_uW^{i,j}_u\big(x_j-x_i\big)q.
\]

\begin{theor}[Adjoint system and gradient]\label{thm:gradient}
Let $g^{(l)}\in C^1(\R)$ with $(g^{(l)})'$ bounded and locally Lipschitz for $l=2,\dots,L$, let $u\in\R^K$ and $x=x(u)$. Let $\lambda\in H^1(0,T;\R^{Nn})$ be the unique solution of the backward adjoint equation
\begin{equation}\label{eq:adjoint}
-\dot\lambda_k=\sum_{i=1}^N DW^{i,k}_u(x_k-x_i)^{\top}\lambda_i-\sum_{j=1}^N DW^{k,j}_u(x_j-x_k)^{\top}\lambda_k+(x_k-z_k),\qquad\lambda_k(T)=0,
\end{equation}
$k=1,\dots,N$, that is, $-\dot\lambda=A(t)^\top\lambda+(x-z)$, $\lambda(T)=0$. Then $\hat J$ is continuously differentiable and
\begin{equation}\label{eq:gradient}
\nabla\hat J(u)=\int_0^TB(t)^\top\lambda(t)\dd t=\int_0^T\sum_{i=1}^N\sum_{j=1}^N\partial_uW^{i,j}_u\big(x_j(t)-x_i(t)\big)^{\top}\lambda_i(t)\dd t.
\end{equation}
\end{theor}

\begin{proof}
By \cref{lem:NN}, $F$ is $C^1$ with $A,B$ continuous in $t$, so \eqref{eq:adjoint} is a linear ODE with continuous coefficients and has a unique solution (\cref{thm:Cara}, applied backward in time). By the differentiability of the control-to-state map, for $q\in\R^K$ the sensitivity $\eta:=x'(u)[q]$ solves
\[
\dot\eta=A(t)\eta+B(t)q,\qquad\eta(0)=0.
\]
Using \eqref{eq:adjoint}, integration by parts, and $\eta(0)=0$, $\lambda(T)=0$, we obtain
\[
\begin{aligned}
\hat J'(u)[q]&=\int_0^T\langle x-z,\eta\rangle\dd t=\int_0^T\langle-\dot\lambda-A^\top\lambda,\eta\rangle\dd t\\
&=-\big[\langle\lambda,\eta\rangle\big]_0^T+\int_0^T\langle\lambda,\dot\eta\rangle\dd t-\int_0^T\langle\lambda,A\eta\rangle\dd t\\
&=\int_0^T\langle\lambda,\dot\eta-A\eta\rangle\dd t=\int_0^T\langle\lambda,Bq\rangle\dd t=\Big\langle\int_0^TB^\top\lambda\dd t,\,q\Big\rangle.
\end{aligned}
\]
Finally, for $h=(h_1,\dots,h_N)$,
\[
\langle\lambda,Ah\rangle=\sum_{i=1}^N\sum_{j=1}^N\big\langle\lambda_i,DW^{i,j}_u(x_j-x_i)(h_j-h_i)\big\rangle,
\]
and collecting the coefficients of $h_k$ (the terms with $j=k$ and with $i=k$, respectively) gives the componentwise form of $A^\top\lambda$ in \eqref{eq:adjoint}.
\end{proof}

\begin{theor}[First-order optimality system]\label{thm:optimality}
Let the assumptions of \cref{thm:gradient} hold and let $\bar u\in\Uad$ be a local solution of \eqref{E:3} with state $\bar x=x(\bar u)$. Then there exists a unique adjoint state $\bar\lambda\in H^1(0,T;\R^{Nn})$ such that
\begin{align}
&\dot{\bar x}_i=\sum_{j=1}^NW^{i,j}_{\bar u}(\bar x_j-\bar x_i),\qquad\bar x_i(0)=z^0_i,\label{eq:OS1}\\
&-\dot{\bar\lambda}_k=\sum_{i=1}^N DW^{i,k}_{\bar u}(\bar x_k-\bar x_i)^{\top}\bar\lambda_i-\sum_{j=1}^N DW^{k,j}_{\bar u}(\bar x_j-\bar x_k)^{\top}\bar\lambda_k+(\bar x_k-z_k),\qquad\bar\lambda_k(T)=0,\label{eq:OS2}\\
&\Big\langle\int_0^T\sum_{i=1}^N\sum_{j=1}^N\partial_uW^{i,j}_{\bar u}(\bar x_j-\bar x_i)^{\top}\bar\lambda_i\dd t,\ u-\bar u\Big\rangle\ge0\qquad\forall u\in\Uad.\label{eq:OS3}
\end{align}
Equivalently, \eqref{eq:OS3} states $\bar u=P_{\Uad}\big(\bar u-\gamma\nabla\hat J(\bar u)\big)$ for every $\gamma>0$, where $P_{\Uad}$ is the componentwise projection onto $[-1,1]$.
\end{theor}

\begin{proof}
We verify the assumptions of \cref{lem:hinze} with
\[
Y:=H^1(0,T;\R^{Nn}),\qquad U:=\R^K,\qquad Z:=L^2(0,T;\R^{Nn})\times\R^{Nn},
\]
\[
e(x,u):=\big(\dot x-F(x,u),\ x(0)-z^0\big),
\]
and $J(x,u):=\frac12\int_0^T|x-z|^2\dd t$.

(A1) $\Uad=[-1,1]^K$ is nonempty ($0\in\Uad$), convex and closed.

(A2) Since $J$ has no explicit dependence on $u$, we have $J_u=0$. To compute the derivative with respect to the state, let $h\in Y$ and $\epsilon\ne0$. Expanding the square,
\[
J(x+\epsilon h,u)=\frac12\int_0^T\Big(|x-z|^2+2\epsilon\langle x-z,h\rangle+\epsilon^2|h|^2\Big)\dd t,
\]
so that
\[
\frac{J(x+\epsilon h,u)-J(x,u)}{\epsilon}=\int_0^T\langle x-z,h\rangle\dd t+\frac\epsilon2\int_0^T|h|^2\dd t.
\]
Letting $\epsilon\to0$ gives
\[
J_x(x,u)[h]=\int_0^T\langle x-z,h\rangle\dd t=\sum_{i=1}^N\int_0^T\langle x_i(t)-z_i(t),h_i(t)\rangle\dd t,
\]
a bounded linear functional depending continuously on $x$. By \cref{lem:NN}, $F$ is $C^1$ with locally Lipschitz derivative; since $H^1(0,T)\hookrightarrow C([0,T])$, the superposition operator $x\mapsto F(x(\cdot),u)$ is continuously Fr\'echet differentiable from $Y\times U$ into $L^2(0,T;\R^{Nn})$. Hence $e$ is continuously Fr\'echet differentiable with
\[
e_x(x,u)[h]=\big(\dot h-D_xF(x,u)h,\ h(0)\big),\qquad e_u(x,u)[q]=\big(-D_uF(x,u)q,\ 0\big).
\]
Componentwise, since the derivative of the linear map $x\mapsto x_j-x_i$ in the direction $h$ is $h_j-h_i$, the chain rule gives
\[
\big[e_x(x,u)[h]\big]_i=\dot h_i-\sum_{j=1}^NDW^{i,j}_u(x_j-x_i)(h_j-h_i),\qquad
\big[e_u(x,u)[q]\big]_i=-\sum_{j=1}^N\partial_uW^{i,j}_u(x_j-x_i)\,q.
\]

(A3) By \cref{thm:wellposed}, $e(x,u)=0$ has a unique solution $x(u)\in C^1([0,T];\R^{Nn})\subset Y$ for every $u\in\R^K$.

(A4) Let $(r,r_0)\in Z$ be arbitrary and consider $e_x(x(u),u)[h]=(r,r_0)$. This is equivalent to the linear initial-value problem
\[
\dot h(t)=A(t)h(t)+r(t),\qquad h(0)=r_0,\qquad A(t):=D_xF(x(u,t),u).
\]
Since $x(u)$ is bounded on $[0,T]$ and $DW^{i,j}_u$ is bounded on bounded sets, there is $C_A>0$ with $|A(t)|\le C_A$ for all $t\in[0,T]$, and \cref{thm:Cara} provides a unique solution $h\in Y$. Moreover,
\[
|h(t)|\le|r_0|+\int_0^t|r(s)|\dd s+C_A\int_0^t|h(s)|\dd s.
\]
By \cref{lem:gronwall} and the Cauchy--Schwarz inequality,
\[
|h(t)|\le\Big(|r_0|+\int_0^T|r(s)|\dd s\Big)e^{C_AT}\le\big(|r_0|+\sqrt T\,\|r\|_{L^2(0,T)}\big)e^{C_AT},
\]
whence $\|h\|_{L^2(0,T)}\le C\big(|r_0|+\|r\|_{L^2(0,T)}\big)$. From $\dot h=Ah+r$ we also obtain
\[
\|\dot h\|_{L^2(0,T)}\le C_A\|h\|_{L^2(0,T)}+\|r\|_{L^2(0,T)}.
\]
Consequently,
\[
\|h\|_{H^1(0,T)}\le C\big(|r_0|+\|r\|_{L^2(0,T)}\big),
\]
which shows that $e_x(x(u),u)^{-1}$ exists and is bounded.

\Cref{lem:hinze} therefore yields $(\bar p,\eta)\in Z^*\cong L^2(0,T;\R^{Nn})\times\R^{Nn}$ with $e_x(\bar x,\bar u)^*(\bar p,\eta)=-J_x(\bar x,\bar u)$, that is,
\begin{equation}\label{eq:weakadj}
\int_0^T\big\langle\bar p,\ \dot h-A(t)h\big\rangle\dd t+\langle\eta,h(0)\rangle=-\int_0^T\langle\bar x-z,h\rangle\dd t\qquad\forall h\in H^1(0,T;\R^{Nn}).
\end{equation}
Choosing $h\in C^\infty_c((0,T);\R^{Nn})$ shows that $\bar p$ has the weak derivative $\dot{\bar p}=(\bar x-z)-A^\top\bar p\in L^2$, hence $\bar p\in H^1$. Integrating by parts in \eqref{eq:weakadj} for general $h$ then gives $\langle\bar p(T),h(T)\rangle+\langle\eta-\bar p(0),h(0)\rangle=0$, so $\bar p(T)=0$ and $\eta=\bar p(0)$. Setting $\bar\lambda:=-\bar p$ we obtain \eqref{eq:OS2}; uniqueness follows from linearity. Finally, since $J_u=0$, the variational inequality of \cref{lem:hinze} reduces to
\[
\big\langle e_u(\bar x,\bar u)^*(\bar p,\eta),\,u-\bar u\big\rangle\ge0\qquad\forall u\in\Uad.
\]
Using the expression for $e_u$, this reads
\[
-\int_0^T\sum_{i=1}^N\sum_{j=1}^N\big\langle\bar p_i(t),\partial_uW^{i,j}_{\bar u}(\bar x_j(t)-\bar x_i(t))[u-\bar u]\big\rangle\dd t\ge0\qquad\forall u\in\Uad,
\]
which is \eqref{eq:OS3} after substituting $\bar p=-\bar\lambda$. The projection characterisation is the standard equivalent form of a variational inequality over a closed convex set.
\end{proof}

\begin{rmk}[Weak form of the optimality system]\label{rmk:weakform}
Multiplying \eqref{eq:weakadj} by $-1$ and writing $\zeta:=-\eta=\bar\lambda(0)$, the adjoint equation \eqref{eq:OS2} can equivalently be stated in the weak form used in \cite{GottlichTotzeck2022}: for all $h=(h_1,\dots,h_N)\in H^1(0,T;\R^{Nn})$,
\[
\int_0^T\sum_{i=1}^N\Big\langle\bar\lambda_i,\ \dot h_i-\sum_{j=1}^NDW^{i,j}_{\bar u}(\bar x_j-\bar x_i)(h_j-h_i)\Big\rangle\dd t+\langle\zeta,h(0)\rangle=\int_0^T\sum_{i=1}^N\langle\bar x_i-z_i,h_i\rangle\dd t.
\]
If the adjoint is instead defined with the opposite sign, $\bar p=-\bar\lambda$, as in the proof above, the variational inequality \eqref{eq:OS3} takes the equivalent form
\[
\Big\langle\int_0^T\sum_{i=1}^N\sum_{j=1}^N\partial_uW^{i,j}_{\bar u}(\bar x_j-\bar x_i)^{\top}\bar p_i\dd t,\ \bar u-u\Big\rangle\ge0\qquad\forall u\in\Uad.
\]
\end{rmk}

\begin{rmk}
\Cref{thm:existence,thm:optimality} together establish the well-posedness of the optimisation problem and its first-order optimality conditions. In words: under the weak regularity assumptions above, the parameter identification problem has at least one optimal parameter set, and every such set is characterised by the coupled forward state equation, backward adjoint equation, and variational inequality.
\end{rmk}

\begin{rmk}\label{rmk:cost}
By \cref{thm:gradient}, the full gradient $\nabla\hat J(u)\in\R^K$ is obtained from one forward solve of \eqref{eq:ode} and one backward solve of \eqref{eq:adjoint}, followed by a quadrature in time. The cost is independent of $K$, whereas finite-difference approximations would require $K+1$ forward solves. For all-to-all interactions each right-hand-side evaluation costs $\mathcal O(N^2)$ network evaluations, and $\mathcal O(N)$ for nearest-neighbour (follow-the-leader) interactions.
\end{rmk}

\begin{rmk}
The sign convention $\lambda=-\bar p$ is chosen so that $\nabla\hat J$ is given by \eqref{eq:gradient} with a plus sign and the optimality condition takes the familiar form $\langle\nabla\hat J(\bar u),u-\bar u\rangle\ge0$. In the interior of $\Uad$ it reduces to $\nabla\hat J(\bar u)=0$.
\end{rmk}

%======================================================================
\section{Adjoint-based stochastic gradient descent}\label{sec:sgd}
%======================================================================

With the theoretical foundation of \cref{sec:ocp} in place, we now introduce an implementation of stochastic gradient descent with adaptive learning rates, based on the methodology of \cite{GottlichTotzeck2022}. The primary goal is to demonstrate an approach to parameter calibration for neural-network-driven models, using interaction modelling within pedestrian and traffic flow systems, and to compare the optimal cost attained by each model. Our code employs a mini-batch, ADADELTA-based update rule combined with variance adaptation strategies, with the aim of optimising the network parameters efficiently while dealing with noisy gradient estimates.

It is very common for the cost functionals of applications to be non-convex, i.e., to possess multiple local minima (\cref{rmk:existence}). When working with large data sets and complex models, it can be challenging for gradient descent algorithms to move on from a local minimum in order to find the global minimum, and evaluating the full gradient in every step is expensive. Techniques such as mini-batch stochastic gradient descent \cite{RobbinsMonro1951,BottouCurtisNocedal2018} have been developed in response. Let the data consist of $m$ reference trajectories (or time slices of a trajectory) $z^{(1)},\dots,z^{(m)}$ with associated costs $\hat J^{(\ell)}$ and $\hat J=\frac1m\sum_\ell\hat J^{(\ell)}$. For a randomly drawn mini-batch $\mathcal B_k\subset\{1,\dots,m\}$ of size $\tilde m$ the estimator
\[
\nabla\hat J^{\tilde m}_k(u):=\frac{1}{\tilde m}\sum_{\ell\in\mathcal B_k}\nabla\hat J^{(\ell)}(u)
\]
is unbiased for $\nabla\hat J(u)$, and each summand is computed by \cref{thm:gradient}. The process works as follows:
\begin{enumerate}[leftmargin=2em,itemsep=0.1em]
\item Start with an initial parameter estimate.
\item Randomly select a mini-batch at each iteration.
\item Calculate the gradient of the cost functional on the selected mini-batch.
\item Update the parameters.
\item Repeat the previous steps until completion, which can take the form of a maximal number of iterations or stability of the cost functional.
\end{enumerate}

Ideally, the learning rate should update automatically during gradient descent, in order to attain faster convergence and to avoid the time and complication of choosing a suitable fixed learning rate. To fine-tune our results dynamically we therefore use the adaptive scheme ADADELTA \cite{Zeiler2012}, which has previously been used for pedestrian interaction networks \cite{GottlichKnapp2020} and uses previous gradients to adjust the learning rate during training. All operations below are componentwise. With decay rate $\rho\in(0,1)$, a small constant $\varepsilon>0$, $\mathrm{RMS}[y]_k:=\sqrt{E[y^2]_k+\varepsilon}$, and the initialisation $E[g^2]_0=0$, $E[\Delta u^2]_0=0$, the iteration reads, for $k=1,2,\dots$,
\begin{align}
g_k&=\nabla\hat J^{\tilde m}_k(u_k),\\
E[g^2]_k&=\rho\,E[g^2]_{k-1}+(1-\rho)\,g_k^2,\\
\Delta u_k&=-\frac{\mathrm{RMS}[\Delta u]_{k-1}}{\mathrm{RMS}[g]_k}\,g_k,\\
E[\Delta u^2]_k&=\rho\,E[\Delta u^2]_{k-1}+(1-\rho)\,\Delta u_k^2,\\
u_{k+1}&=P_{\Uad}\big(u_k+\Delta u_k\big).
\end{align}
Here $u_k$ denotes the $k$-th iterate and $\Delta u_k=u_{k+1}-u_k$ the increment between two consecutive parameter estimates (before projection), i.e., the step change of the parameters in the $k$-th iteration. The averaged squared gradients $E[g^2]$ and the averaged squared parameter updates $E[\Delta u^2]$ are initialised with zero at iteration $0$ to begin the process.

The second line is the running average of the squared gradients at iteration $k$. The decay rate $\rho$ determines how much weight is assigned to previous estimates, and $g_k$ is the mini-batch gradient of the cost functional at the current iterate. The gradient is squared componentwise to measure its magnitude, which yields a rolling average of squared gradients.

The third line defines how the parameters are updated. In the ratio, the numerator $\mathrm{RMS}[\Delta u]_{k-1}$ acts as a scale that measures the step size through the past squared updates and gives the update the correct units, which removes the need for a global learning rate. The denominator $\mathrm{RMS}[g]_k$ normalises the gradient and reduces noise that could otherwise throw off the results. The ratio is multiplied by the current gradient estimate. The overall purpose is an adaptive step size that improves the convergence of the stochastic gradient descent.

The fourth line updates the running average of the squared parameter changes with the squared update $\Delta u_k^2$ and the decay factor $\rho$, so that the algorithm keeps a record of the magnitudes of previous parameter updates. Finally, the projection $P_{\Uad}$ keeps the iterates admissible and is motivated by the projection form of \cref{thm:optimality}. The complete method is summarised in \cref{alg:sgd}. In all experiments we use $\rho=0.95$, $\varepsilon=10^{-6}$ and $12{,}000$ iterations.

\begin{algorithm}[t]
\caption{Adjoint-based projected mini-batch ADADELTA}\label{alg:sgd}
\begin{algorithmic}[1]
\Require data $z^{(1)},\dots,z^{(m)}$, initial guess $u_1\in\Uad$, batch size $\tilde m$, $\rho\in(0,1)$, $\varepsilon>0$, $k_{\max}$
\State $E[g^2]_0\gets0$, $E[\Delta u^2]_0\gets0$
\For{$k=1,2,\dots,k_{\max}$}
  \State draw a mini-batch $\mathcal B_k\subset\{1,\dots,m\}$ with $|\mathcal B_k|=\tilde m$
  \For{$\ell\in\mathcal B_k$}
     \State solve the state equation \eqref{eq:ode} forward on $[0,T]$ with $u_k$ and initial data $z^{(\ell)}(0)$
     \State solve the adjoint equation \eqref{eq:adjoint} backward on $[0,T]$ with $\lambda(T)=0$
     \State evaluate $\nabla\hat J^{(\ell)}(u_k)$ by quadrature of \eqref{eq:gradient}
  \EndFor
  \State $g_k\gets\frac1{\tilde m}\sum_{\ell\in\mathcal B_k}\nabla\hat J^{(\ell)}(u_k)$
  \State $E[g^2]_k\gets\rho E[g^2]_{k-1}+(1-\rho)g_k^2$; \quad $\Delta u_k\gets-\sqrt{E[\Delta u^2]_{k-1}+\varepsilon}\,\big/\sqrt{E[g^2]_k+\varepsilon}\;g_k$
  \State $E[\Delta u^2]_k\gets\rho E[\Delta u^2]_{k-1}+(1-\rho)\Delta u_k^2$; \quad $u_{k+1}\gets P_{\Uad}(u_k+\Delta u_k)$
\EndFor
\State \Return $u_{k_{\max}+1}$
\end{algorithmic}
\end{algorithm}

\begin{figure}[t]
    \centering
    \includegraphics[width=0.8\linewidth]{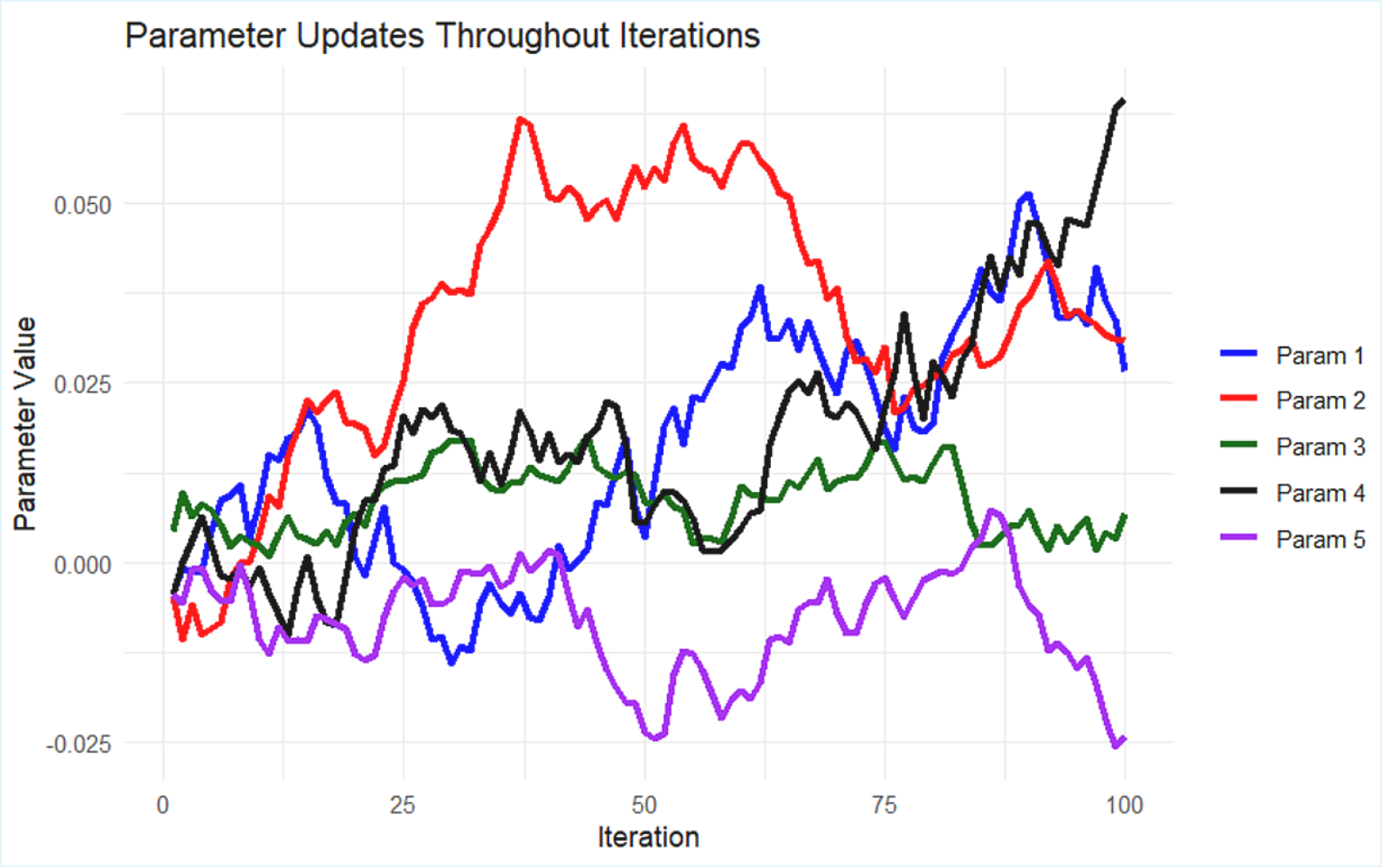}
    \caption{Evolution of five parameters during the first $100$ iterations of the stochastic gradient scheme with ADADELTA.}
    \label{fig:SGD}
\end{figure}

\Cref{fig:SGD} illustrates the behaviour of the scheme over the first $100$ iterations for five representative parameters. The parameter values remain within the narrow range $[-0.025,0.075]$, and the accumulated squared updates $E[\Delta u^2]_k$ are of order $10^{-5}$ to $10^{-4}$, which indicates that the iterates are either approaching a stationary point or oscillating mildly around it, as expected for noisy mini-batch gradients. Some parameters display initial oscillations before plateauing within a narrow band, while others explore a larger range before settling around a baseline. These noise and convergence patterns give an indication of the effectiveness of adaptive learning rates.

The strengths of this methodology begin with the robustness provided by adaptivity: the adaptive step size, built on running estimates of the squared gradients, provides a robust approach to parameter optimisation without manual tuning of a learning rate. In addition, the implementation reproduces an established adaptive stochastic gradient method, ADADELTA, which supports the fidelity of the computations. Furthermore, the parameter evolution plot gives a clear view of the gradual convergence process, which is useful for diagnostics such as detecting overfitting or divergence.

%======================================================================
\section{Traffic dynamics}\label{sec:traffic}
%======================================================================

The adjoint-based stochastic gradient method applies to both parameter identification and neural-network-based models, and vehicle and pedestrian trajectory archives such as \cite{trafficdynamics,Pedynamics} provide natural data sources. Below, two established microscopic traffic models are examined, each of which has been implemented in order to compute its optimal cost and to obtain a visual representation of its strengths and weaknesses: a follow-the-leader model of LWR type \cite{HoldenRisebro2018} and a neural-network model in the spirit of \cite{GottlichTotzeck2022}. Vehicles are indexed so that vehicle $N$ is the leader and vehicle $i<N$ follows vehicle $i+1$. The two models read
\begin{align}
\frac{\dd}{\dd t}x_i&=W^{i,i+1}_u\,\frac{x_{i+1}-x_i}{L},\quad i=1,\dots,N-1,\qquad\frac{\dd}{\dd t}x_N=v_0,\label{eq:lwr}\\
\frac{\dd}{\dd t}x_i&=W^{i,i+1}_u\big(x_{i+1}-x_i\big),\quad i=1,\dots,N-1,\qquad\frac{\dd}{\dd t}x_N=v_0,\label{eq:anntraffic}
\end{align}
respectively, where $L$ is the vehicle length and the leader moves with constant velocity $v_0$. Both are nearest-neighbour instances of \eqref{eq:ode}. In all traffic experiments $N=5$ vehicles are simulated on $[0,T]$ with $T=30$\,s. Beyond these models, a new second-order traffic model is introduced and contrasted with them with respect to the unique attributes it possesses and the trade-offs that come with its design.

\subsection{Follow-the-leader model of LWR type}

In this car-following model each vehicle adjusts its speed according to the distance to the vehicle ahead, normalised by the vehicle length; the velocity of each vehicle is scaled by the coefficient $W^{i,i+1}_u$, which reflects the drivers' reactions. We capture the individual behaviour of each vehicle while also computing the optimal cost of the calibration. Microscopic models of this type converge to the LWR equation in the many-particle limit \cite{HoldenRisebro2018}. We use the model to test whether physical parameters can be recovered from trajectories: reference trajectories were generated from \eqref{eq:lwr} with known parameters $u=(L,v_0)=(5,30)$, and these parameters were then estimated solely from the trajectory data by minimising the discrepancy $\hat J$ between simulated and observed positions. The overall goal was to assess whether the methodology can reliably infer the system parameters, which supports its potential for real-world applications in which the direct measurement of parameters is challenging.

\begin{figure}[t]
    \centering
    \includegraphics[width=0.8\linewidth]{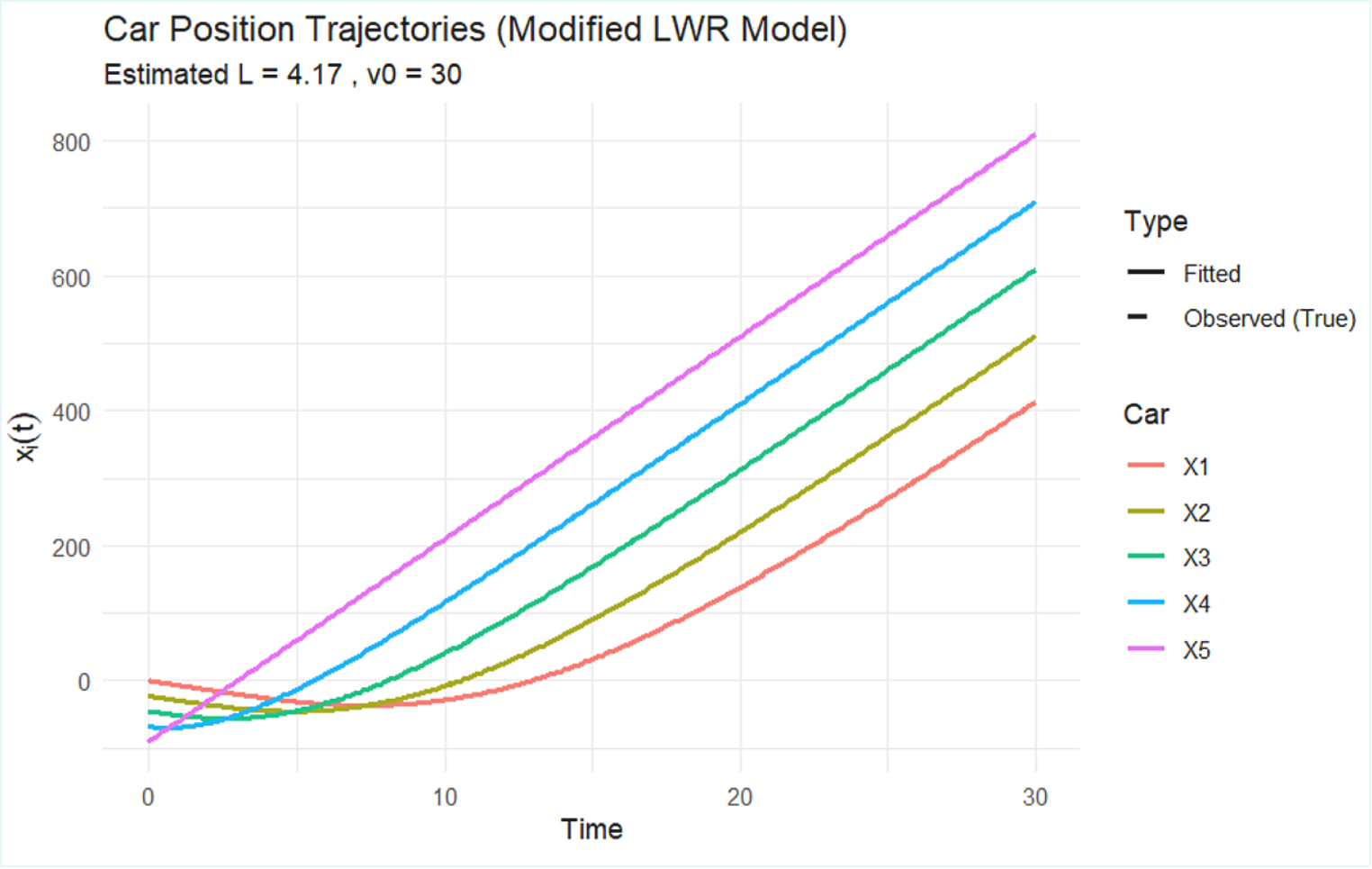}
    \caption{Follow-the-leader model of LWR type: observed (true) and fitted trajectories of five vehicles; estimated parameters $L=4.17$, $v_0=30$.}
    \label{fig:LWRPlot}
\end{figure}

\Cref{fig:LWRPlot} shows that over the $30$\,s horizon the fitted trajectories of all five vehicles follow a uniform shape, with nearly constant spacing from the middle of the time interval onward. The estimated values are close to the true parameters, especially considering the noise-free nature of the data. The leader velocity $v_0=30$ is recovered almost exactly, which demonstrates the accuracy of the approach for the dynamic parameters, while the vehicle length is underestimated ($L=4.17$ instead of $5$, an error of about $0.83$\,m). This discrepancy may stem from the initial guess, from weak sensitivity of the trajectories with respect to $L$ (only the ratio $W^{i,i+1}_u/L$ enters \eqref{eq:lwr}, which limits identifiability), or from the non-convexity of the cost landscape. The cost decreases from $144.56$ to $J=102.04$ (\cref{CostTable102}), which corresponds to a root-mean-square deviation of $\bar d\approx1.17$\,m by \eqref{eq:rmsd}; the table shows a significant reduction of the cost as the model refines itself over the iterations.

The model has several limitations. The follow-the-leader dynamics assume a specific interaction law and identical vehicles, while real traffic exhibits large variability in vehicle characteristics and external influences. The reference data are noise free, whereas real measurements contain errors. Finally, convergence to a global minimum depends on the initial guess and on the optimiser. The methodology can be enhanced by incorporating noise or measurement errors, by extending the model to account for more individual vehicle attributes, and by using real traffic data to evaluate how practical the method is.

\subsection{Neural-network traffic model}

In the second model the length normalisation of the previous model is removed. Instead, the primary focus is an interaction law inspired by neural networks within the framework of a differential equation: $W^{i,i+1}_u$ is a feed-forward network that maps the headway $x_{i+1}-x_i$ to a velocity, embedded in \eqref{eq:anntraffic}, and models how much influence neighbouring vehicles have on each other. The network output can be positive or negative, representing accelerating and decelerating influences, respectively.

\begin{figure}[t]
    \centering
    \includegraphics[width=0.8\linewidth]{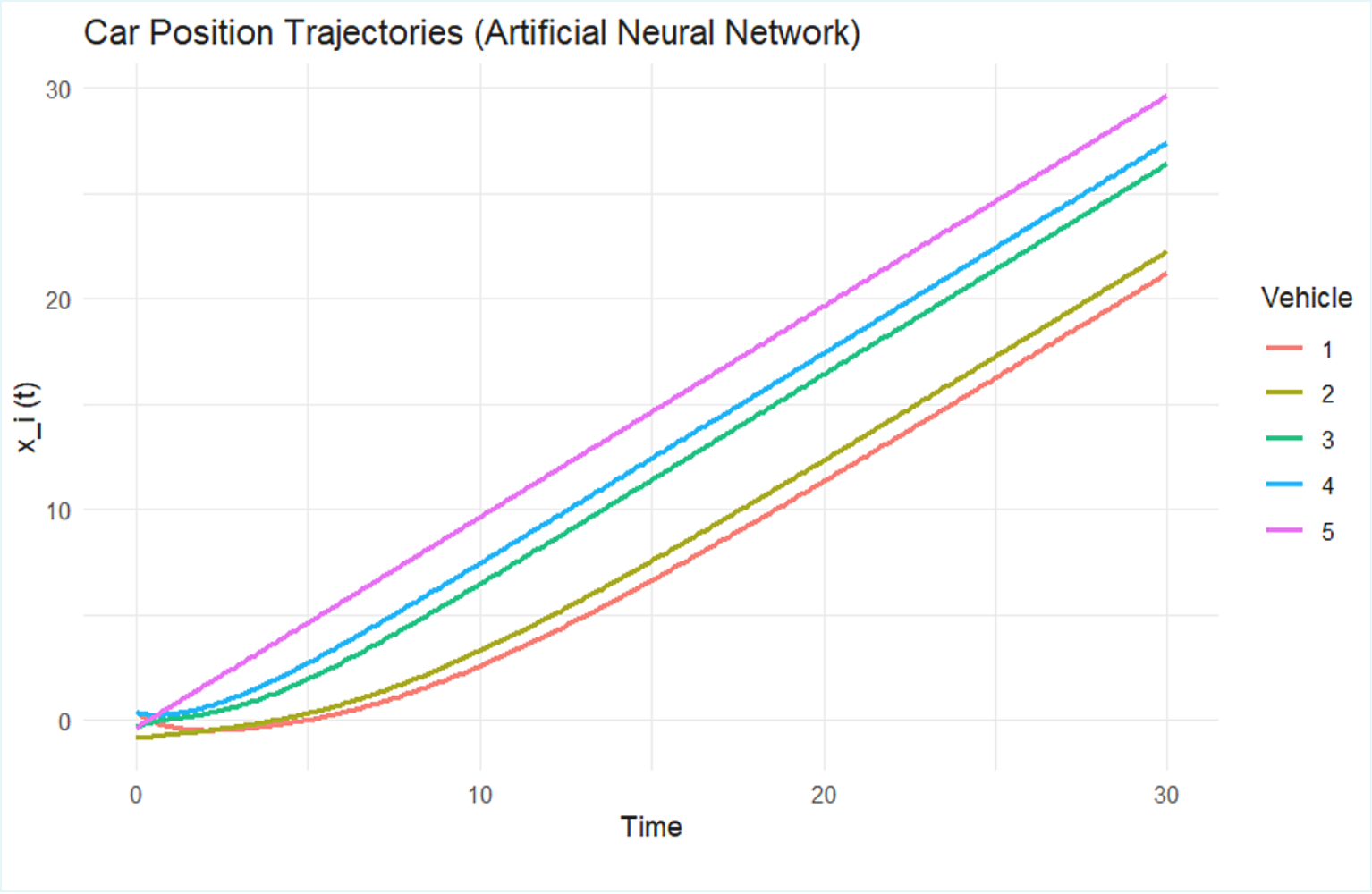}
    \caption{Neural-network traffic model: trajectories of five vehicles.}
    \label{fig:ArtNNPlot}
\end{figure}

\begin{table}[hbt]
  \centering
  \begin{minipage}{6.0in}
    \centering
    \caption{Table representing the gradual change in cost for the neural-network traffic model}
    \label{jctab}
    \begin{tabular}{|c||c|}    \hline
      Iteration & Cost \\ \hline \hline
      1 & 187.26 \\ \hline
      1000 & 183.21 \\ \hline
      2000 & 176.37 \\ \hline
      4000 & 167.20 \\ \hline
      8000 & 158.17 \\ \hline
      10000 & 151.43 \\ \hline
      12000 & 146.02 \\ \hline
    \end{tabular}
  \end{minipage}
\end{table}

The use of a neural network yields a more realistic representation of traffic dynamics than the LWR-type model, as shown in \cref{fig:ArtNNPlot}. The model does not try to mirror a fixed velocity and vehicle length; instead, more natural human biases appear, and vehicles are not always uniformly spaced. During the run, the model progressively adjusts itself to better fit the training data, and the cost decreases monotonically from $187.26$ to $J=146.02$ (\cref{jctab}), i.e.\ $\bar d\approx1.40$\,m. The decrease has not yet levelled off after $12{,}000$ iterations, which suggests that longer runs would continue to refine the model towards the optimal cost.

Initial vehicle states are sampled randomly from distributions that emulate different traffic configurations, such as uniform spacing, in order to diversify the training scenarios. This allows us to capture a large range of potential traffic patterns. The simulated trajectories reproduce several phenomena observed in real traffic. First, vehicle clusters form and dissolve depending on the local interaction strengths, mirroring congestion and its dissolution. Second, the leader's constant velocity acts as a boundary condition whose influence propagates through the platoon; an erratically driving leader would produce an irregular distribution of the vehicles behind it. Third, emergent oscillations eventually give way to overall stability: depending on the weights, the system exhibits stable flow or wave-like oscillations reminiscent of stop-and-go traffic. Random initial weights can induce repeated oscillations, which illustrates how driver variability affects traffic flow.

Randomly initialised weights also provide insight into the impact of different driver behaviours. A positive response to a widening gap indicates that the follower speeds up, whereas a negative response models braking or cautious driving, with a reduction of speed as the gap shrinks. This allows a heterogeneous traffic environment to be depicted, and the simulated trajectories show how different interactions produce distinct traffic patterns, such as stop-and-go oscillations, jams, or smooth passage. These results suggest that even with randomly initialised weights it is possible to obtain a system that reflects complex and realistic patterns.

Using one interaction law for all drivers oversimplifies real-world driver behaviour; it lacks complexity compared with actual data, but can be useful for capturing larger trends. The constant-velocity boundary condition cannot represent traffic conditions such as accidents or inflow and outflow at certain points. Finally, the model does not include lane changes, road conditions, or traffic signals, so it represents a generic stretch of road.

Several improvements could raise the sophistication of the model. The framework could be extended to capture multiple traffic lanes simultaneously. The weights could be updated from traffic data, for instance by making $W_u$ dependent on the state or on time in order to adapt to different driving behaviours. Implementing these changes should enhance the predictive power of the model and provide a better understanding for traffic management.

\subsection{A new second-order traffic model}\label{sec:newtraffic}

We now propose a second-order model that combines driver intent, control, and learned interactions with repulsion and velocity alignment:
\begin{align}
\dot x_i&=v_i,\label{eq:newtraffic1}\\
\dot v_i&=\frac1{m_i}\Big(\alpha_i\big(v_i^\des(x_i)-v_i\big)+\beta_ic_i(t)+\sum_{j\ne i}W^{i,j}_u\Big[-\frac{x_j-x_i}{|x_j-x_i|^{\phi_1}}+\phi_2\tanh(v_j-v_i)\Big]\Big),\label{eq:newtraffic2}
\end{align}
with the feedback control $c_i(t)=k_i(t)\big(v_i^\des-v_i(t)\big)$ and the time-dependent gain $k_i(t)=k_i^{\mathrm{base}}\big(1-e^{-rt}\big)^{n_k}$.

By Newton's second law, the bracket on the right-hand side of \eqref{eq:newtraffic2} is the net force acting on vehicle $i$, and dividing by the mass $m_i$ yields its acceleration. The term $\alpha_i(v_i^\des(x_i)-v_i)$ is a goal-seeking force that represents the driver's intent to travel at a desired, possibly position-dependent, velocity, as in the optimal velocity and intelligent driver models \cite{Bando1995,Treiber2000}. The term $\beta_ic_i$ is a control input representing throttle, brake, or a driver-assistance system; the base gain $k_i^{\mathrm{base}}$ is a vehicle-specific scaling that defines the maximal strength of the feedback, which is ramped up smoothly with rate $r$ and exponent $n_k$. The sum is the interaction force: the weights $W^{i,j}_u$ are retained from \cite{GottlichTotzeck2022}, while the repulsion term $-(x_j-x_i)/|x_j-x_i|^{\phi_1}$ prevents collisions by acting more strongly on nearby vehicles, with strength governed by the exponent $\phi_1$, and the alignment term $\phi_2\tanh(v_j-v_i)$ acts as a damper that drives vehicles towards matching velocities. The combination of alignment, attraction and repulsion is inspired by flocking models \cite{CuckerSmale2007}. Among the benefits of this model are parameter learning and function approximation: all coefficients and exponents can be learned from data in the same way as network weights. The model is also more expressive and adaptable than the previous models, as seen in \cref{fig:NewTrafPlot}.

\begin{rmk}\label{rmk:singular}
The repulsion term is singular at collisions and therefore not globally Lipschitz, so \cref{thm:wellposed} does not apply verbatim. However, as long as the inter-vehicle distances stay bounded below by some $\delta>0$, the right-hand side of \eqref{eq:newtraffic1}--\eqref{eq:newtraffic2} is locally Lipschitz in $(x,v)$ and continuously differentiable in all parameters, so local existence, uniqueness and the adjoint representation of \cref{thm:gradient} carry over on such time intervals. Since the repulsion grows as distances shrink, collisions are not observed in our simulations.
\end{rmk}

\begin{figure}[t]
    \centering
    \includegraphics[width=0.8\linewidth]{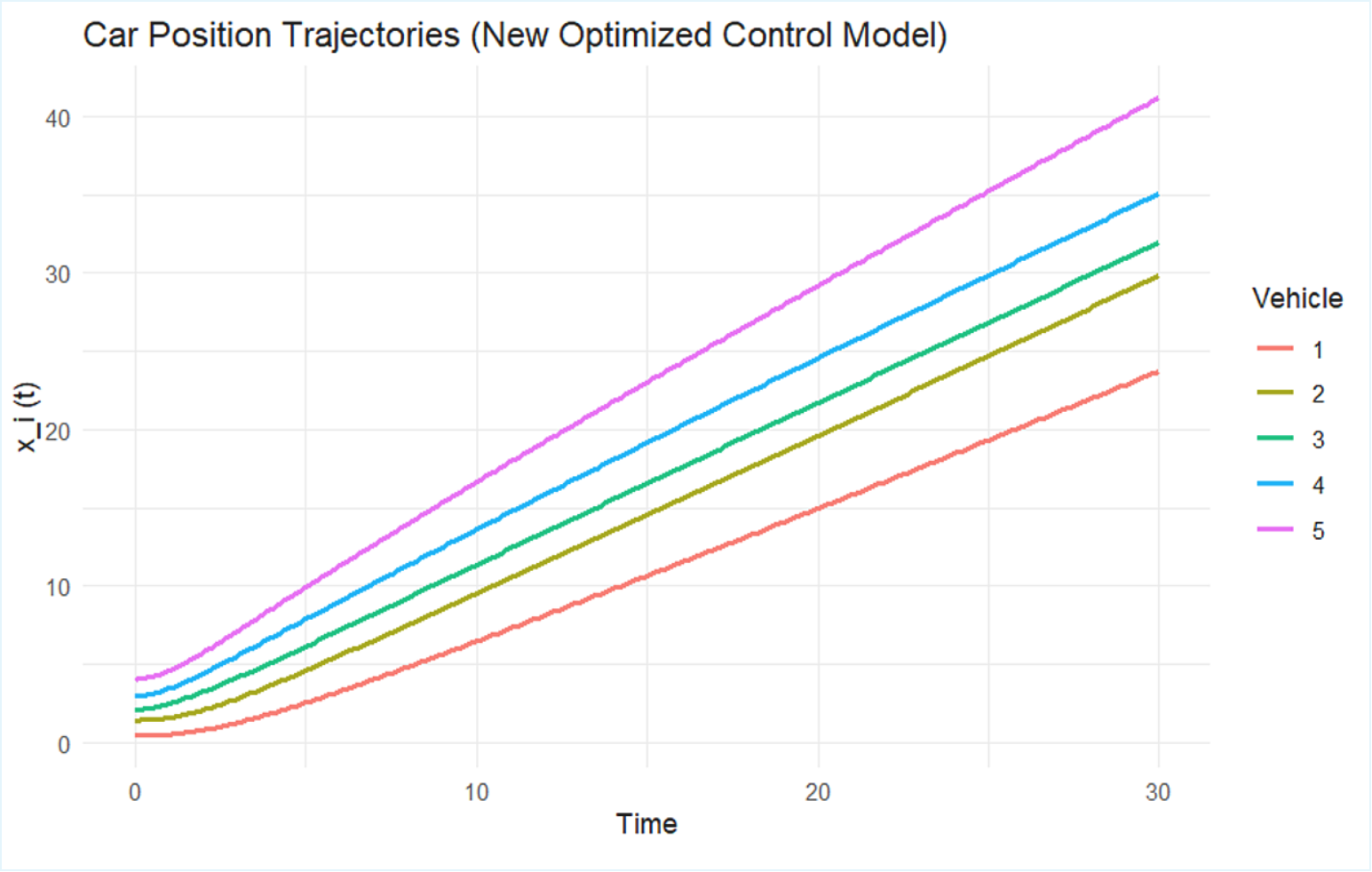}
    \caption{New second-order traffic model: trajectories of five vehicles.}
    \label{fig:NewTrafPlot}
\end{figure}

Although \cref{fig:NewTrafPlot} resembles the previous plots, the model has several attributes that make it valuable in specific scenarios. Unlike the LWR-type model, it accounts for driver intent through the goal-seeking term, and as a second-order model it evolves positions and velocities, whereas both established models prescribe the velocity directly. Individual vehicles exhibit the most diverse velocities and the least uniform behaviour among all traffic examples. The feedback gain $k_i(t)$ individualises the control response, and the interaction law combines spacing, alignment, and control instead of spacing alone. This richer structure comes at a price: the cost decreases from $2211.74$ to $J=1186.02$ (\cref{CostTable1186}), a reduction of $46\%$, but the final value is considerably higher than for the first-order models and corresponds to $\bar d\approx3.98$\,m.

The choice of model therefore depends on its purpose. For aggregate trends over large samples the established first-order models may be preferable because they are cheap to calibrate and fit well. For moderate sample sizes in which individual behaviour matters, the new model offers interpretable, vehicle-specific parameters. There is no such thing as a perfect model, and deciding what can and cannot be sacrificed is a large part of the trade-off in choosing which model to work with.

%======================================================================
\section{Crowd dynamics}\label{sec:crowd}
%======================================================================

We next apply the same methodology to crowds, with pedestrians in place of vehicles. This introduces environmental forces from walls and obstacles and shifts the focus from vehicles to individual humans. As for the traffic models, we again employ the stochastic gradient algorithm. Specifically, we have implemented the social force model \cite{HelbingMolnar1995} and a neural-network model \cite{GottlichKnapp2020,GottlichTotzeck2022} in order to compute their optimal costs; both are second-order models in $\R^2$:
\begin{align}
\dot x_i&=v_i,\qquad\dot v_i=v_i^D+\frac1{mN}\sum_{j=1}^NF_{i,j}+\frac1{mN_w}\sum_{k=1}^{N_w}F_{i,w_k},\label{eq:sfm}\\
\dot x_i&=v_i,\qquad\dot v_i=v_i^D+\frac1{mN}\sum_{j=1}^NW^{i,j}_u\big(x_i-x_j,v_i-v_j\big)+\frac1{mN_w}\sum_{k=1}^{N_w}W^{i,w}_u\big(x_i-w_k,v_i\big),\label{eq:pednn}
\end{align}
respectively, where $w_1,\dots,w_{N_w}$ are discretisation points of the walls and $m$ is the pedestrian mass. The relaxation towards the desired motion is
\begin{equation}\label{eq:vD}
v_i^D=\frac1\tau\big(v_i^\des(t)-v_i(t)\big),\qquad v_i^\des(t)=\frac{x_D-x_i(t)}{|x_D-x_i(t)|}\,|v_i(t)|,
\end{equation}
with relaxation time $\tau$ and destination $x_D$. The quantities of particular importance are the relaxation time $\tau$, the number of wall points $N_w$, the number of pedestrians $N$, and the force constants $A,B,k,\kappa$ introduced below. In all pedestrian experiments $N=8$ pedestrians are simulated over $T=120$\,s, with individual speeds drawn randomly from $[1.25,2.00]$\,m/s.

\subsection{Social force model}

In the social force model of Helbing and Moln\'ar each individual moves towards a desired destination and adjusts its velocity under the influence of other pedestrians and of walls. With pedestrian radius $r$, $h(s):=\max\{s,0\}$, and
\[
d_{ij}=|x_i-x_j|,\qquad n_{ij}=\frac{x_i-x_j}{d_{ij}},\qquad t_{ij}=\big(-n_{ij}^2,\,n_{ij}^1\big),\qquad\Delta v^t_{ji}=(v_j-v_i)\cdot t_{ij},
\]
the pedestrian--pedestrian and pedestrian--wall forces are
\begin{align}
F_{i,j}&=\Big(A\exp\Big(\frac{2r-d_{ij}}{B}\Big)+k\,h(2r-d_{ij})\Big)n_{ij}+\kappa\,h(2r-d_{ij})\,\Delta v^t_{ji}\,t_{ij},\label{eq:Fij}\\
F_{i,w}&=\Big(A\exp\Big(\frac{r-d_{iw}}{B}\Big)+k\,h(r-d_{iw})\Big)n_{iw}-\kappa\,h(r-d_{iw})\,(v_i\cdot t_{iw})\,t_{iw},\label{eq:Fiw}
\end{align}
where $d_{iw}$, $n_{iw}$, $t_{iw}$ are defined analogously with respect to a wall point. The exponential term models social repulsion, while the terms with $k$ and $\kappa$ model body compression and sliding friction upon contact, as in the extension of the social force model by Helbing, Farkas and Vicsek \cite{HelbingFarkasVicsek2000}.

We describe pedestrian dynamics by a computational model that integrates boundary conditions, social interactions, and environmental influences, and takes pedestrian repulsion, wall constraints, and velocity limitations into account. The primary goal is to evaluate how a small-scale pedestrian model reproduces observed trajectories, and to analyse the impact of the individual forces on the trajectories. Pedestrians are agents subject to a repulsion force that prevents collisions; the forces depend on relative velocities and on proximity to other agents, following the exponential characteristics of the social force model. Walls act as upper and lower boundaries and limit the region in which agents can move. The physical parameters (interaction range, repulsion strength, body force, friction) are calibrated by \cref{alg:sgd}. Pedestrians are initialised on a grid with prescribed spacing, their destinations are drawn randomly from a prescribed region in order to simulate a realistic journey, and each individual has a maximal speed drawn randomly within physical bounds, in contrast to the fixed speeds used in \cite{TordeuxEtAl2019}.

\begin{figure}[t]
    \centering
    \includegraphics[width=0.8\linewidth]{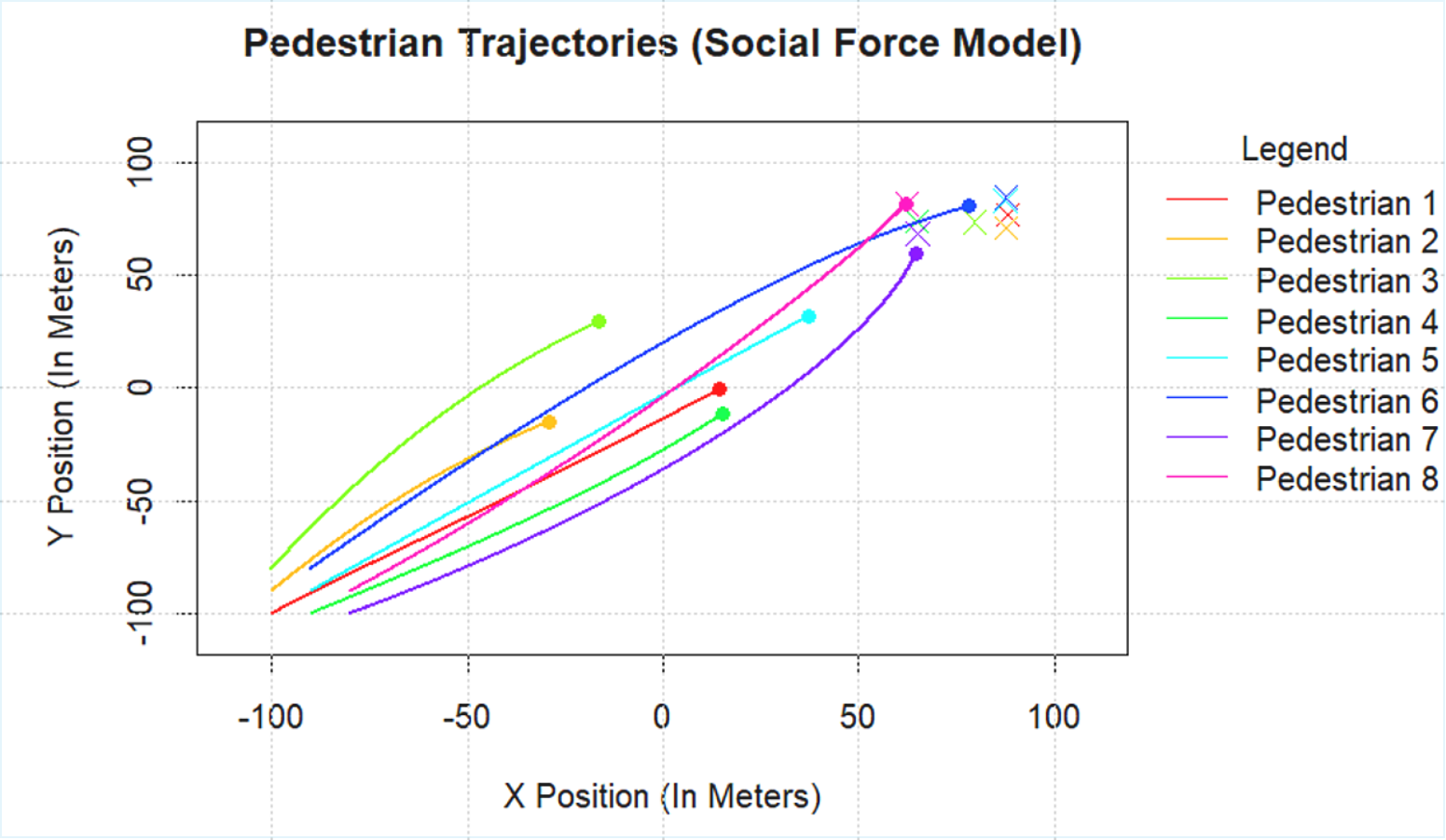}
    \caption{Social force model: trajectories of eight pedestrians (circles: final positions; crosses: destinations).}
    \label{fig:SocForcePlot}
\end{figure}

\begin{figure}[t]
    \centering
    \includegraphics[width=0.8\linewidth]{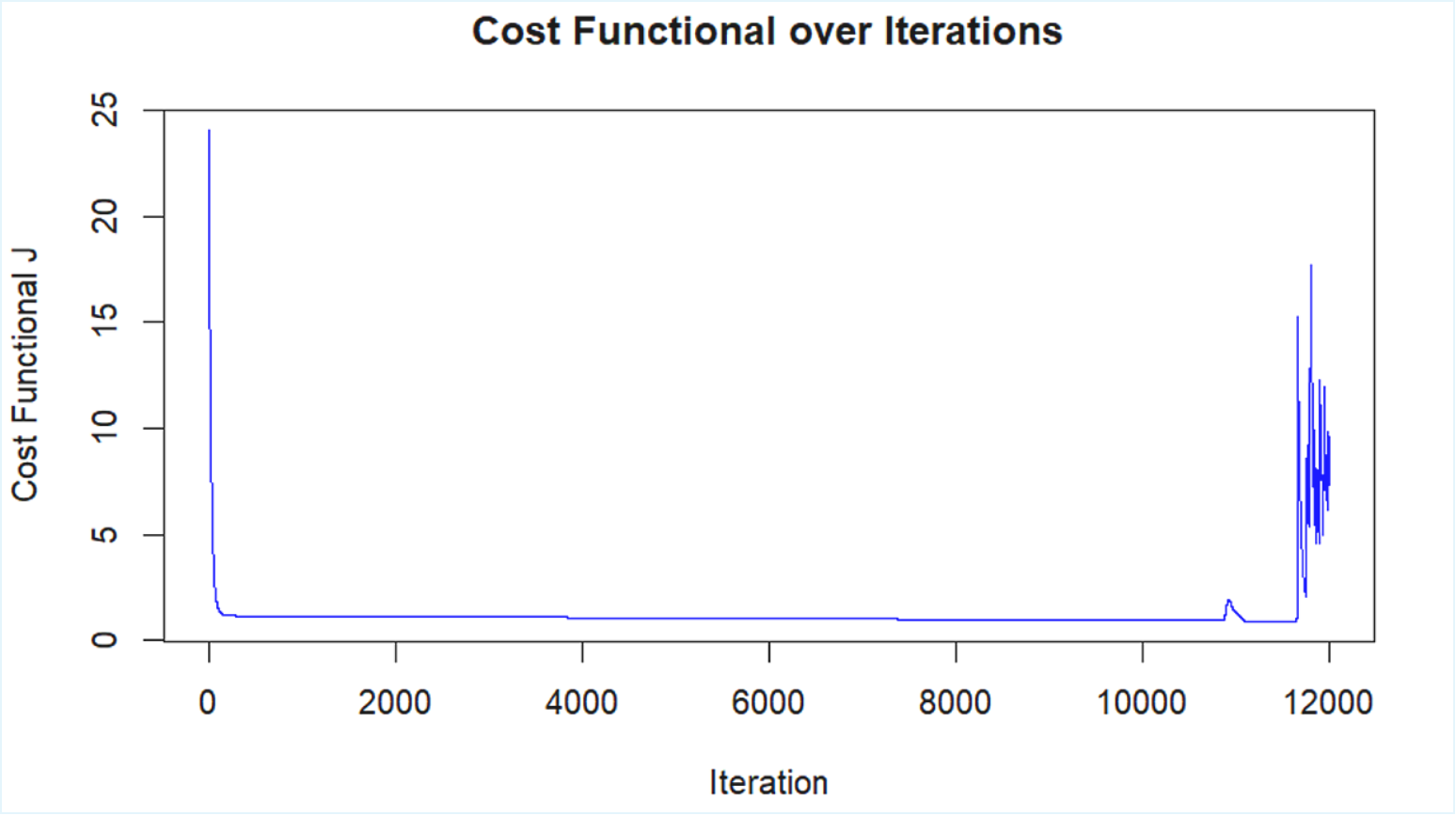}
    \caption{Evolution of the cost functional for the social force model.}
    \label{fig:CostFuncPlot}
\end{figure}

\Cref{fig:SocForcePlot} shows that all pedestrians move towards their destinations under the influence of the forces and the boundaries. They share space while avoiding collisions and respecting the walls; the social repulsion disperses the group, eliminates overlap, and lets the same locations be occupied at different times. Individuals adopt distinct, often non-straight, paths to avoid collisions. Within the time horizon some pedestrians reach their destinations and the others make significant progress. Since pedestrians move much more slowly than vehicles, the costs are considerably smaller than in the traffic experiments: the final cost is $J=7.31$, corresponding to $\bar d\approx0.12$\,m. Owing to the complexity of the crowd models, the cost estimate oscillates more strongly than in the traffic case. As shown in \cref{fig:CostFuncPlot}, the cost drops rapidly during the first iterations and then stays almost constant, but it exhibits pronounced oscillations towards the end of the run (after roughly $11{,}000$ iterations). Such late spikes are typical of adaptive stochastic methods when the gradient estimates become small relative to $\varepsilon$; early stopping, a smaller $\varepsilon$, or gradient clipping are simple remedies.

There are several avenues for improving this model. Each pedestrian could be given individual personal-space forces to simulate a more realistic variety of behaviours. The scenario could be expanded by simulating a larger crowd and by letting pedestrians interact with dynamic environmental elements rather than only static walls. Furthermore, the model could be refined by calibrating the interaction forces on real-world data, adding another level of sophistication to the results.

\subsection{Neural-network pedestrian model}

The neural-network approach replaces the analytical pairwise forces by a data-driven model that approximates them as functions of relative positions and velocities, see \eqref{eq:pednn}. As a function approximator, the network can capture complex, nonlinear interaction patterns between pedestrians and between pedestrians and obstacles. Following \cite{GottlichKnapp2020}, trajectory data are preprocessed by slicing them into short time intervals in which the number of people is fixed; this lets the model focus on short-term dynamics and simplifies the learning process, because the force function can be assumed constant or smoothly varying within these short intervals, and it allows the slices to serve as the samples of the mini-batches in \cref{alg:sgd}.

\begin{figure}[t]
    \centering
    \includegraphics[width=0.8\linewidth]{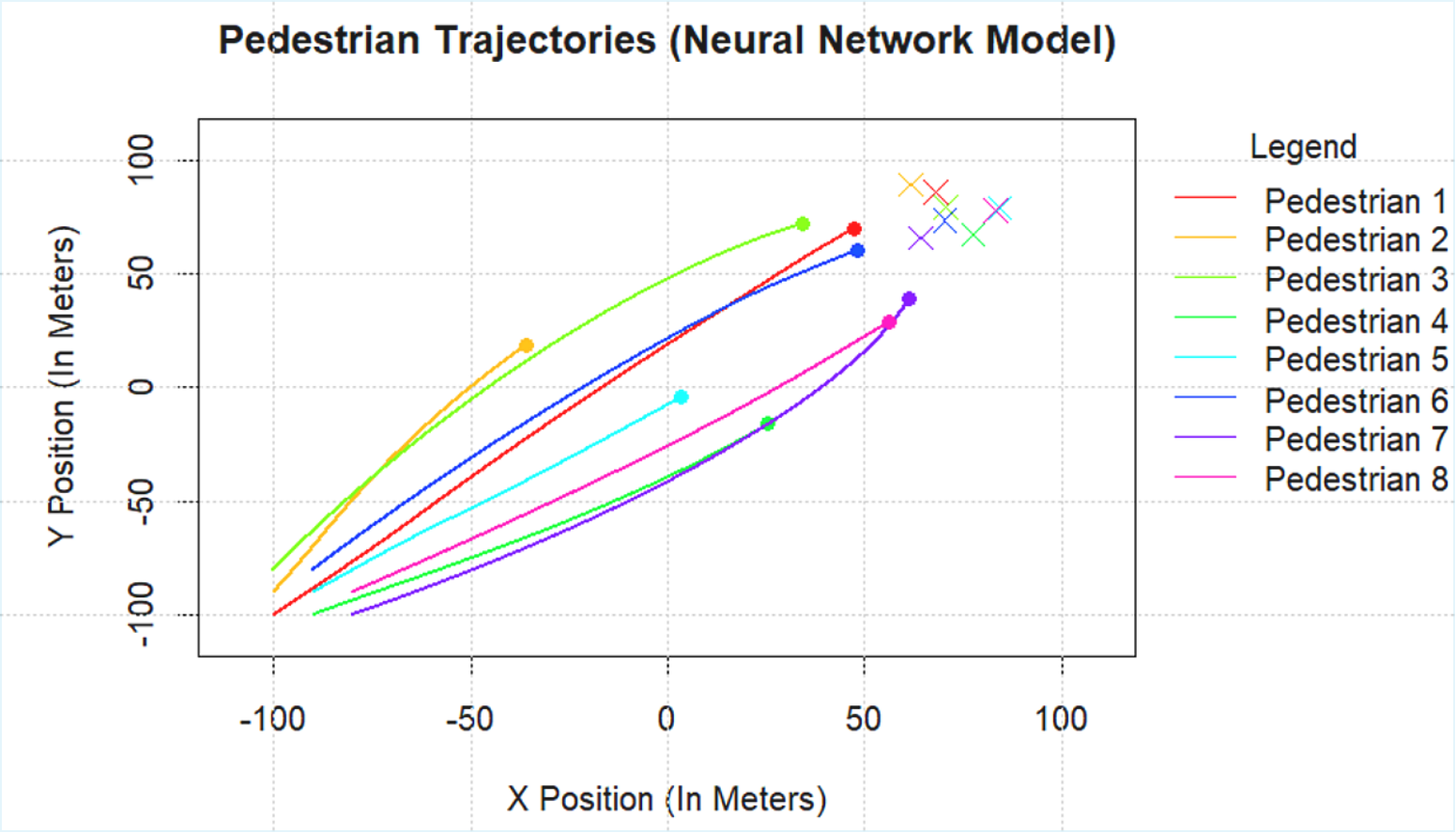}
    \caption{Neural-network pedestrian model: trajectories of eight pedestrians.}
    \label{fig:NNPeoplePlot}
\end{figure}

\Cref{fig:NNPeoplePlot} shares the overall picture of the previous experiment: individuals follow their own efficient routes towards their destinations. Compared with the social force model, the paths are more clearly separated, without the large cluster observed before. All pedestrians make significant progress, but none reaches its destination within the time horizon. A likely cause is the wall at $y=90$, whose repulsion is strong enough to slightly reduce the velocity of the pedestrians heading towards it, just enough to prevent them from completing their route. As stated above, all pedestrian models use the same time horizon of $120$\,s and the same range of speeds. The cost decreases from $23.93$ to $J=5.86$ (\cref{CostTable5.86}), a reduction of $75.5\%$ and the lowest final cost of all models considered, corresponding to $\bar d\approx0.11$\,m.

The approach can be improved by adapting the interaction ranges to proximity and movement context, for example with an architecture that processes local and global positional information separately in order to distinguish immediate collision avoidance from navigation around walls and other people. Encoding wall geometries or boundary conditions directly as network inputs would enable the model to learn more nuanced and consistent responses to boundaries, further improving obstacle avoidance and adherence to environmental constraints. This would result in a more flexible and accurate representation of pedestrian movement, particularly in complex environments.

\subsection{A new crowd model}\label{sec:newcrowd}

We propose a new crowd model that shares several features with the new traffic model:
\begin{equation}\label{eq:newcrowd}
\dot x_i=v_i,\qquad\dot v_i=\frac1{m_i}\Big(\alpha_i(x_i)\big(v_i^\des-v_i\big)+\sum_{j\ne i}W^{i,j}_u\,\Phi_u\big(x_j-x_i,v_j-v_i\big)-D_iv_i+\beta_ic_i\Big).
\end{equation}
The first term is a state-dependent motivation: the gain $\alpha_i(x_i)$ adjusts the motivation of each individual according to its location. The sum contains the interaction weights and a neural interaction function $\Phi_u$ that outputs a force vector depending on the relative state and can learn attraction and repulsion from data. The dissipative term $-D_iv_i$ and the control input are common in multi-agent coordination \cite{RenBeardAtkins2005}: the damping coefficient $D_i$ represents friction or resistance; by removing energy from the system it accounts for human traits such as fatigue or reluctance to speed up and prevents unbounded acceleration. In the control term $\beta_ic_i$, the coefficient $\beta_i$ scales the influence of an external input $c_i$, such as a guidance signal or an announcement; in a crowd without external stimuli, $c_i\equiv0$. Since $\Phi_u$ is a network in the sense of \cref{def:NN} and the remaining terms are linear in the velocity, \cref{thm:wellposed,thm:gradient} extend to \eqref{eq:newcrowd} whenever $\alpha_i$ is Lipschitz.

\begin{figure}[t]
    \centering
    \includegraphics[width=0.8\linewidth]{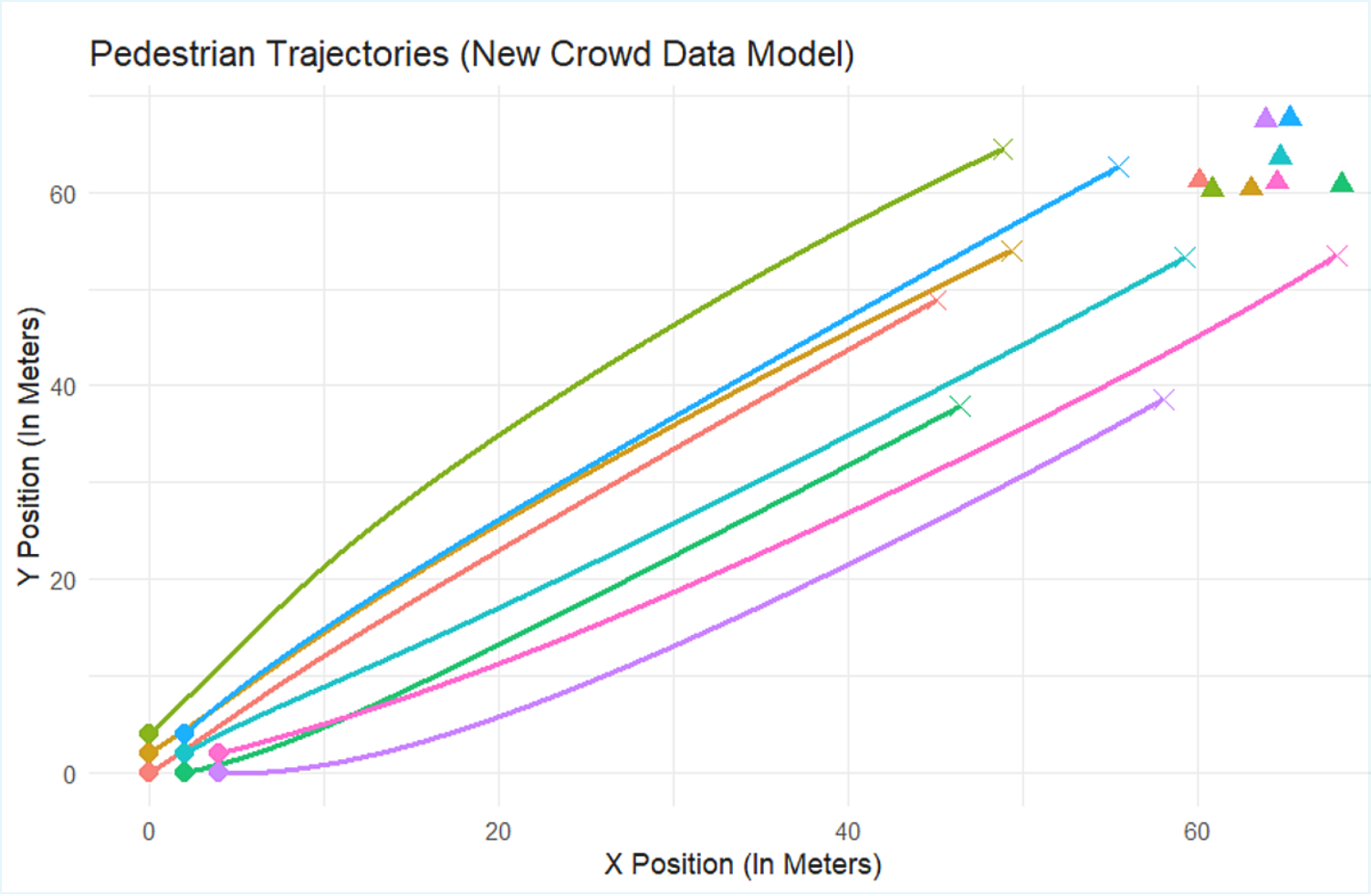}
    \caption{New crowd model: trajectories of eight pedestrians.}
    \label{fig:NewCrowdPlot}
\end{figure}

\begin{table}[hbt]
  \centering
  \begin{minipage}{6.0in}
    \centering
    \caption{Table representing the gradual iterative change in cost for the new crowd model}
    \label{jctab77}
    \begin{tabular}{|c||c|}    \hline
      Iteration & Cost \\ \hline \hline
      1 & 32.93 \\ \hline
      1000 & 24.62 \\ \hline
      2000 & 22.32 \\ \hline
      4000 & 19.68 \\ \hline
      6000 & 17.89 \\ \hline
      8000 & 16.71 \\ \hline
      10000 & 15.90 \\ \hline
      12000 & 15.32 \\ \hline
    \end{tabular}
  \end{minipage}
\end{table}

As shown in \cref{fig:NewCrowdPlot}, no individual reaches its destination within the time horizon. The conditions of this experiment were kept as close as possible to the previous two for the sake of comparison, so the most likely cause is the combination of damping and position-dependent motivation. Other causes, such as strong repulsion or the absence of external guidance, are less likely because the plot shows neither clusters nor blockages.

The model differs from the previous ones through motivation, external control, and heterogeneity. Instead of a fixed goal-velocity term it uses position-dependent motivation; for instance, individuals may be less motivated in congested areas and more motivated near exits. While the earlier models are purely self-driven or reactive, the new model admits external stimuli such as guidance signals, so that assisted and unassisted navigation can be simulated within one framework. Individual masses, motivations, damping and control give each person a distinct response to the same situation, and these dynamic traits allow a much more realistic picture of how individuals react. The cost is more than halved, from $32.93$ to $J=15.32$ (\cref{jctab77}), corresponding to $\bar d\approx0.18$\,m, and is still decreasing at the end of the run. The initial cost being more than twice the final one shows that the model has made significant progress in adjusting its fit over the course of the run.

%======================================================================
\section{Comparative discussion}\label{sec:discussion}
%======================================================================

\Cref{tab:summary} summarises the calibration results of all six models. The deviation $\bar d$ is computed from the final cost by \eqref{eq:rmsd}.

\begin{table}[t]
\centering
\caption{Summary of the calibration results. $J_1$: cost at the first iteration; $J_{\mathrm{final}}$: cost after $12{,}000$ iterations; $\bar d$: root-mean-square deviation \eqref{eq:rmsd}.}\label{tab:summary}
\begin{tabular}{llcccrrc}
\toprule
Application & Model & Order & $N$ & $T$ [s] & $J_1$ & $J_{\mathrm{final}}$ & $\bar d$ [m]\\
\midrule
Traffic & LWR-type follow-the-leader & 1 & 5 & 30 & 144.56 & 102.04 & 1.17\\
Traffic & Neural network & 1 & 5 & 30 & 187.26 & 146.02 & 1.40\\
Traffic & New second-order model & 2 & 5 & 30 & 2211.74 & 1186.02 & 3.98\\
\midrule
Crowd & Social force & 2 & 8 & 120 & -- & 7.31 & 0.12\\
Crowd & Neural network & 2 & 8 & 120 & 23.93 & 5.86 & 0.11\\
Crowd & New crowd model & 2 & 8 & 120 & 32.93 & 15.32 & 0.18\\
\bottomrule
\end{tabular}
\end{table}
% NOTE from Dao: \bar d is computed consistently as sqrt(2J/(N T)). This reproduces the traffic
% values of the thesis (1.17, 1.40, 3.98 m) exactly. For the pedestrian models the thesis reported
% 0.35, 0.26 and 0.45 m, which are not consistent with one formula; Edgar: please verify N and T used in
% the pedestrian cost against the code.

Several observations can be made.
\begin{enumerate}[label=(\roman*),leftmargin=2em]
\item \emph{Monotone decrease.} In all experiments with recorded cost histories the cost decreases monotonically over the reported checkpoints, with reductions of $22\%$ (neural-network traffic), $29\%$ (LWR type), $46\%$ (new traffic), $53\%$ (new crowd) and $76\%$ (neural-network crowd). Several histories are still decreasing at the end of the run, so longer runs are expected to improve the fits further.
\item \emph{Neural interactions fit pedestrian data best.} The neural-network pedestrian model attains the lowest final cost and the largest relative reduction, which supports the use of learned interaction forces where analytical force laws are uncertain.
\item \emph{Richer models are not automatically better fits.} Both new second-order models attain higher final costs than the simpler models of the same application. With more parameters and more nonlinear terms the cost landscape becomes harder to explore, some parameters are only weakly identifiable (compare the LWR-type test, where only $W/L$ enters the dynamics), and convergence within a fixed iteration budget is slower. The benefit of the new models lies in interpretable, individual-level parameters (motivation, gain, damping, control) rather than in a smaller misfit.
\item \emph{Robustness of the optimiser.} ADADELTA removes the need to tune a learning rate, but the late cost spikes for the social force model (\cref{fig:CostFuncPlot}) show that adaptive methods can become unstable when the gradient estimates are small; monitoring, early stopping or clipping should be used in practice.
\item \emph{Computational cost.} By \cref{rmk:cost} each iteration requires one forward and one backward solve per sample, independently of the number of network weights. Follow-the-leader models scale linearly in $N$, all-to-all crowd models quadratically, which suggests neighbour lists or cut-off radii for large crowds.
\end{enumerate}

The present study also has limitations that should be kept in mind. The experiments use small numbers of agents, a single calibration run per model, and, in the LWR-type test, synthetic noise-free reference data. A systematic validation would include repeated runs with different random seeds, noisy data, held-out test trajectories, and calibration on real trajectory archives such as \cite{trafficdynamics,Pedynamics}.

%======================================================================
\section{Future directions}\label{sec:future}
%======================================================================

\emph{Traffic.} The traffic model offers many avenues for improvement, especially in goal navigation, behavioural extensions, and interactions with a dynamic environment. At present, each vehicle is driven by a fixed desired velocity; vehicles are unaware of their destination and do not adapt their speed to the remaining distance or to any urgency. Assigning each vehicle a predefined route, possibly with intermediate waypoints, together with time-sensitive instructions that mimic off-ramps or traffic signs, would significantly enhance the realism of the model.

Individual drivers also differ markedly from those around them. Adding driver states such as stress, fatigue, or caution could be used to prioritise certain goals for each vehicle; for example, a driver who has to weigh safety against speed on a given route would make the simulation considerably more dynamic. This would also increase the chance of observing potential outliers on the road, which is useful for safety considerations.

The driving environment is currently empty and static. Implementing guard rails, crosswalks, weather conditions, and similar features would bring the results of the model closer to real-world data. Isolating the effect of a single potential hazard from that of several hazards is a further question that a future model might address. In this way, much more individualised simulations can be created while maintaining a common underlying description of the drivers on the road.

\emph{Crowds.} For the crowd model, progress could be made through more advanced position tracking, perception constraints, terrain boundaries, or cognitive modelling. At present, the model uses a desired velocity as control target. In the real world, however, individuals head for more complex target positions, such as exits or restrooms. Distinct motivations could be assigned to particular positions to represent urgency more accurately. Furthermore, goals shift over time, so desired coordinates that change after a certain amount of time are another aspect worth investigating.

Regarding perception, the present interaction law assumes that all individuals observe each other equally, which is not realistic. Instead, interactions can be filtered through a field of view, so that lines of sight depend on the distance and relative angle to others; people tend to react to others in front of them much more than to those behind them. Realism could be increased further by letting environmental obstacles, such as a hallway or a fountain, restrict what people can see. It may also be beneficial to introduce a memory of previously seen individuals.

The new crowd model currently operates in a fairly open space with few barriers. Physical constraints such as curbs and furniture can reroute an individual's path to varying degrees. It may also help to introduce bottlenecks at which individuals are forced to backtrack and reassess their movement. Individual obstacles can moreover carry their own repulsion forces instead of the static wall force currently used.

Finally, cognitive layering can introduce significantly more complex crowd dynamics. If individuals remember positions, velocities, or interactions, the simulations become more faithful to real data. With larger numbers of agents than in our examples, individuals can also be assigned to groups, such as families or friends; travelling in small groups changes how people traverse certain paths, particularly when the identities of those involved (children, adults, or elderly people) are taken into account.

There is considerable weight on many factors in the creation and use of a dynamic model. Pedestrians and vehicles each face distinct situations and forces that may or may not be taken into account, and the inclusion of each individual aspect determines how useful the model truly is for the situation at hand.

\emph{Methodology.} From the mathematical point of view, it is natural to study the mean-field limit $N\to\infty$ of the calibration problem and its connection to kinetic and macroscopic models \cite{PareschiToscani2013,Carrillo2010}, to quantify the uncertainty of the identified parameters in a Bayesian setting \cite{Stuart2010,CorbettaEtAl2015}, and to compare ADADELTA with other adaptive or derivative-free methods.

%======================================================================
\section{Conclusions}\label{sec:conclusion}
%======================================================================

We have presented a unified optimal control framework for calibrating traffic and pedestrian models from trajectory data, covering both classical parameter identification and the learning of interaction forces by neural networks. For the network-driven particle system we established well-posedness, Lipschitz stability with respect to initial data and network weights, existence of optimal parameters, and a first-order optimality system with an explicit backward adjoint equation. The resulting gradient representation underlies a projected mini-batch ADADELTA scheme whose cost per iteration does not depend on the number of parameters. Using this scheme we calibrated two established models for each application and proposed two new second-order models that incorporate driver or pedestrian intent, learned interactions, repulsion and alignment or damping, and external control. The experiments show consistent cost reduction for all models, with the neural-network pedestrian model achieving the best fit, while the new models offer individualised and interpretable behaviour at the price of a harder calibration problem. Validation on real data, larger populations, and the extensions outlined above are the next steps.

\section*{Acknowledgements}
The authors thank the Department of Mathematics and Statistics at San Diego State University for its support. Parts of this work are based on the Master's thesis of the first author at San Diego State University.

\section*{Statements and declarations}
\noindent\textbf{Funding.} The research of D.~Nguyen was supported by start-up funds from San Diego State University.

\noindent\textbf{Competing interests.} The authors declare that they have no competing interests.

\noindent\textbf{Author contributions.} E.L.: Methodology, Software, Investigation, Visualization, Writing -- original draft. D.N.: Conceptualization, Methodology, Formal analysis, Supervision, Writing -- review \& editing.

\noindent\textbf{Data and code availability.} The trajectory archives used as data sources are publicly available \cite{trafficdynamics,Pedynamics}. The code used to generate the results is available from the corresponding author upon reasonable request.

\appendix
\section{Additional cost tables}\label{app:tables}

\begin{table}[H]
  \centering
  \begin{minipage}{6.0in}
    \centering
    \caption{Table representing the gradual change in cost for the LWR-type simulation}
    \label{CostTable102}
    \begin{tabular}{|c||c|}    \hline
      Iteration & Cost \\ \hline \hline
      1 & 144.56 \\ \hline
      1000 & 134.40 \\ \hline
      2000 & 125.62 \\ \hline
      4000 & 118.59 \\ \hline
      8000 & 112.13 \\ \hline
      10000 & 106.81 \\ \hline
      12000 & 102.04 \\ \hline
    \end{tabular}
  \end{minipage}
\end{table}

\begin{table}[H]
  \centering
  \begin{minipage}{6.0in}
    \centering
    \caption{Table representing the gradual change in cost for the new traffic model}
    \label{CostTable1186}
    \begin{tabular}{|c||c|}    \hline
      Iteration & Cost \\ \hline \hline
      1 & 2211.74 \\ \hline
      1000 & 1903.21 \\ \hline
      2000 & 1776.37 \\ \hline
      4000 & 1574.20 \\ \hline
      8000 & 1448.16 \\ \hline
      10000 & 1291.30 \\ \hline
      12000 & 1186.02 \\ \hline
    \end{tabular}
  \end{minipage}
\end{table}

\begin{table}[H]
  \centering
  \begin{minipage}{6.0in}
    \centering
    \caption{Table representing the gradual change in cost for the pedestrian neural-network model}
    \label{CostTable5.86}
    \begin{tabular}{|c||c|}    \hline
      Iteration & Cost \\ \hline \hline
      1 & 23.93 \\ \hline
      1000 & 17.24 \\ \hline
      2000 & 15.65 \\ \hline
      4000 & 13.80 \\ \hline
      6000 & 10.42 \\ \hline
      8000 & 7.93 \\ \hline
      10000 & 5.98 \\ \hline
      12000 & 5.86 \\ \hline
    \end{tabular}
  \end{minipage}
\end{table}

\end{document}